\documentclass{ws-jml}

\newcommand{\N}{{\mathbb N}}
\newcommand{\Z}{\mathbb Z}
\newcommand{\Q}{\mathbb Q}
\newcommand{\R}{\mathbb R}
\newcommand{\lh}{\operatorname{lh}}
\newcommand{\conc}{{}^{\smallfrown}}

\newcommand{\Bai}{\ensuremath{ \N^{ \N }}}

\newcommand{\dom}{\operatorname{dom}}

\newcommand{\res}{\restriction}
\newcommand{\es}{\emptyset}

\newcommand{\rest}[1]{ |_{#1}}

\newcommand{\rao}{\rightarrow}

\newcommand{\lgrao}{\longrightarrow}
\newcommand{\Lglrao}{\Longleftrightarrow}
\newcommand{\lao}{\leftarrow}
\newcommand{\lgto}{\longmapsto}

\newcommand{\sub}{\subseteq}

\DeclareMathOperator{\Emb}{Emb}
\DeclareMathOperator{\Epi}{Epi}
\DeclareMathOperator{\Hrk}{rk_H}
\DeclareMathOperator{\FinSum}{\mathsf{FinSum}}
\DeclareMathOperator{\Supp}{Supp}

\newcommand{\StS}{\mathsf{StS}}
\newcommand{\WO}{\mathsf{WO}}
\newcommand{\scat}{\mathsf{Scat}}
\newcommand{\lin}{\mathsf{Lin}}

\newcommand{\Pieces}{\mathbf{Pieces}}
\newcommand{\DefPieces}{\mathbf{DefPieces}}
\newcommand{\strong}{\leq_{s}}
\newcommand{\inj}{\leq_{i}}

\begin{document}

\markboth{Riccardo Camerlo, Rapha\"{e}l Carroy, Alberto Marcone}
{When embeddability and epimorphism agree}

\catchline{}{}{}{}{}

\title{Linear orders: when embeddability and epimorphism agree}

\author{Riccardo Camerlo}
\address{Dipartimento di matematica, Universit\`a di Genova, Via Dodecaneso 35, 16146 Genova --- Italy; \email{camerlo@dima.unige.it}}

\author{Rapha\"el Carroy}
\address{Kurt G\"{o}del Research Center, W\"{a}hringer Strasse 25, 1090 Wien --- Austria; \email{raphael.carroy@univie.ac.at}}

\author{Alberto Marcone}
\address{Dipartimento di Scienze Matematiche, Informatiche e Fisiche, Universit\`a di Udine, Via delle Scienze 208, 33100 Udine --- Italy; \email{alberto.marcone@uniud.it}}

\maketitle

\begin{history}
\received{(Day Month Year)}
\revised{(Day Month Year)}
\end{history}

\begin{abstract}
When a linear order has an order preserving surjection onto each of its
suborders we say that it is \emph{strongly surjective}. We prove that the
set of countable strongly surjective linear orders is a
$\check{\mathrm{D}}_2(\mathbf{\Pi}^1_1)$-complete set. Using hypotheses
beyond ZFC, we prove the existence of uncountable strongly surjective
orders.
\end{abstract}

\keywords{Linear orders, epimorphisms, effective descriptive set theory, strongly surjective linear orders}

\ccode{Mathematics Subject Classification 2010: 03E15, 06A05, 03E65}

\section{Introduction}
There are two natural ways of comparing a pair of linear orders $L$ and $M$:
embeddability and epimorphism. We write $L \inj M$ when there is an order
preserving injection, also called an \emph{embedding}, from $L$ to $M$.
Similarly, $L \strong M$ stands for the existence of an order preserving
surjection, also called an \emph{epimorphism}, from $M$ onto $L$. The
equivalence relation associated to $\strong$ is written $\equiv_s$.

Using the axiom of choice, $L \strong M$ implies $L \inj M$, but the
embeddability relation $\inj$ is in general weaker than the relation
$\strong$ induced by epimorphisms. For example, the ordinal number $\omega $
embeds into $\omega+1$, but there is no epimorphism from $\omega +1$ onto
$\omega $.

There are however linear orders $M$ for which the relations $L \inj M$ and $L
\strong M $ turn out to be equivalent. The ordinals satisfying this property
have been characterized in \cite{cacama} (see Theorem \ref{ssordinal} below).
The aim of this article is to study the class of orders $M$ for which the
two notions coincide.

For the purpose of this paper, when talking about a linear order, we will always
assume that it is non-empty; in particular, if no contrary mention is given, when
a linear order is written as a sum $\sum_{i\in I} L_i$, all the summands are
assumed to be non-empty.

We are now ready to give our main definition.

\begin{definition} \label{thedefinition}
A linear order $M$ is \emph{strongly surjective} if, for any linear order
$L$, $L \inj M$ implies $L \strong M$; equivalently, if $M$ surjects
order-preservingly onto each of its suborders.
\end{definition}

The following characterization of strongly surjective ordinals is Corollary
29 of \cite{cacama}.

\begin{theorem}\label{ssordinal}
An ordinal is strongly surjective if and only if it is a finite multiple of
an indecomposable countable ordinal, that is, if it is of the form
$\omega^{\alpha }m$, for some $\alpha<\omega_1$ and $m>0$.
\end{theorem}

The rationals are also strongly surjective: indeed by Proposition
16(1) in \cite{cacama} $L \strong \Q$ for every countable linear order $L$. Up to
$\equiv_s$, $\Q$ is the only countable non-scattered strongly surjective
order (recall that $L$ is \emph{scattered} if $\Q \not \inj L$): see
Proposition \ref{nonscat} below.\medskip

Our main result is the following classification of the descriptive complexity
of the set of countable strongly surjective linear orders:

\begin{theorem} \label{mainintroduction}
The set of countable strongly surjective orders
is $\check{\mathrm{D}}_2(\mathbf{\Pi}^1_1)$-complete.
\end{theorem}

Here $\check{\mathrm{D}}_2(\mathbf{\Pi}^1_1)$ is the class of sets which are
union of an analytic and a coanalytic set. The set we are interested in
belongs to this class because the set of scattered strongly surjective orders
is $\mathbf{\Pi}^1_1$, while the set of non-scattered strongly surjective
orders is $\mathbf{\Sigma}^1_1$. In fact they are both complete in their
respective classes (Corollary \ref{cor:scatstrongsurj} and Proposition
\ref{basichardness2}).

Our proof of the upper bound for scattered strongly surjective orders makes
an essential use of both effective descriptive set theory and the fact that
$\leq_s$ is a well quasi-order on the countable linear orders. The latter is
the main theorem of \cite{Landra1979} and \cite{cacama}.

Even if the study of the first two levels of the projective hierarchy is a
long-standing topic, examples of sets that are true $\boldsymbol \Delta^1_2$
(that is, $\boldsymbol \Delta^1_2$ but neither analytic nor coanalytic) are
very rare. The interest in these sets has recently been rekindled by
Fournier's study of the difference hierarchy of co-analytic sets
(\cite{fournier}). However, as far as we know, the set of countable strongly
surjective orders is the first concrete example of a
$\check{\mathrm{D}}_2(\mathbf{\Pi}^1_1)$-complete set that is not made so by
design. Furthermore, two natural examples of sets which are
complete in the dual class $\mathrm{D}_2(\mathbf{\Pi}^1_1)$ (consisting of
the intersections of an analytic and a coanalytic set) were found in
\cite{distances} and \cite{AC}.\bigskip

Here is the plan of the paper.\smallskip

In Section \ref{strongsurjectivity}, we prove some basic properties of
strongly  surjective linear orders, and we present a useful way of defining
epimorphisms by pieces, that we use throughout the paper.\smallskip

We start studying the descriptive complexity of the set of countable strongly
surjective linear orders in Section \ref{lower bounds}. The set $\lin$ of all
linear orders $\leq_K$ on a subset $K$ of $\N$ is Polish as it is a closed subspace of
$2^{ \N \times \N }$. We then call $\StS$ the set of strongly surjective
orders in $\lin$. Definition \ref{thedefinition} immediately gives an upper
bound, $\StS$ being indeed a $\Pi^1_2$ subset of $\lin$. We prove in this
section that $\StS$ is $\check{\mathrm{D}}_2(\mathbf{\Pi}^1_1)$-hard (Theorem
\ref{main}). Our proof uses a study of the powers of $ \Z $ and we notably
prove that $ \Z^K$ is strongly surjective for all countable
$K$.\smallskip

In Section \ref{analysis} we show that for any countable scattered linear
order $K$, there is a $\Delta^1_1(K)$ function that maps a linear order $L$
to an epimorphism from $K$ to $L$ when it exists, and to the refusing symbol
$\bot$ otherwise (Theorem \ref{theorem:boreluniformization}). As a corollary
we get a $\check{\mathrm{D}}_2(\mathbf{\Pi}^1_1)$ definition of $\StS$
(Corollary \ref{cor:upperbound}). This completes the proof of Theorem
\ref{mainintroduction}.\smallskip

Finally, Section \ref{unctbless} deals with uncountable linear orders. We
first prove that many concrete (e.g.\ $\R$, $\R \setminus \Q$, their finite
products, and also $\R^\N$, $\Q^\N$, $2^\alpha$ for $\alpha<\omega_1$) are
not strongly surjective, leaving open the problem of the provability in ZFC
of the existence of an uncountable strongly surjective linear order. By
contrast, we prove the existence of uncountable strongly surjective orders
assuming either PFA (Theorem \ref{underPFA}) or the existence of what we call
a Baumgartner tree (Theorem \ref{prop:underdiamond}). The latter hypothesis
is connected to the principle $\Diamond^+$, and thus orthogonal to PFA.

We conclude by discussing some problems that remain open and suggest new
lines of research.

\subsection{About notations}
Variable symbols $K,L,M$ always stand for linear orders. $L^\star$ stands for
the reverse of the linear order $L$. We call \emph{equimorphism} the
equivalence relation $\equiv_i$ associated to $\inj$, and we use the symbol
$\simeq$ to denote isomorphism. The notation for operations such as sums and
products on linear orders is standard; a reference is \cite{Rosens1982}. In
particular, $\sum_{i\in I} L_i$ is the sum ordered by $I$ of disjoint copies
of each $L_i$, in other words $\bigcup_{i\in I}\{i\}\times L_i$ ordered lexicographically. The multiplicative notation $LK$ stands for $K$ copies of $L$, i.e. $\sum_{k \in K} L$.

Given an order $(K,\leq_K)$, and $p\in K$, define $({\lao},p]_K$ as $\{n\in
K\mid n\leq_K p\}$ and order it with the order induced by $\leq_K$. Define in
a similar fashion the orders $({\lao},p)_K$, $(p,q)_K$, $(p,q]_K$, $[p,q)_K$,
$[p,q]_K$, $(p,{\rao})_K$ and $[p,{\rao})_K$. We allow the notation $[p,q]_K$
when $p=q$ as well, letting then $[p,p]_K=\{ p\} $. All these sets (including
$({\lao},{\rao})_K = K$) will be called \emph{intervals}.

A subset $K'$ of $K$ is \emph{convex} when $x,y \in K'$ and $x \leq_K y$
imply $[x,y]_K \subseteq K'$ (so every interval is convex, but not all convex
sets are intervals).

We call $\Emb(L,K)$ the set of all embeddings from $L$ to $K$, and
$\Epi(L,K)$ the set of all epimorphisms from $K$ onto $L$.


\section{Strong surjectivity} \label{strongsurjectivity}
We begin by stating some basic properties of strongly surjective orders.

\begin{proposition} \label{basics}~
\begin{arabiclist}
\item A linear order $L$ is strongly surjective if and only if $L^\star$ is.
\item If $L$ is strongly surjective and $M \inj L \strong M$, then $M$ is
    strongly surjective and $L \equiv_s M$.
\item If $L$ and $M$ are strongly surjective and $L \equiv_i M$, then $L
    \equiv_s M$.
\end{arabiclist}
\end{proposition}

\begin{proof}
(1) is obvious. (2) Let $K \inj M$. Since $M \inj L$ and $L$ is strongly
surjective, there is an epimorphism $L\to K$.
As there is also an epimorphism $M \to L$, this yields $K \strong
M$. (3) follows from the definition of strongly surjective.
\end{proof}

Part (3) of Proposition \ref{basics} states that in any class of equimorphism
there is at most one $\equiv_s$-class of strongly surjective orders. However,
not every class of equimorphism contains a strongly surjective order. Indeed,
for an ordinal number $\alpha $ the classes of equimorphism, isomorphism and
bi-epimorphism coincide. So if $\alpha$ is not of the form given by Theorem
\ref{ssordinal}, its equimorphism class does not contain any strongly
surjective order.

The results of \cite{cacama} easily yield the following
characterizations of countable strongly surjective linear orders that are not
scattered:

\begin{proposition}\label{nonscat}
Let $L$ be a countable non-scattered linear order. The following are
equivalent:
\begin{arabiclist}
  \item $L$ is strongly surjective;
  \item $\Q \strong L$;
  \item $L$ has no initial or final segment which is scattered.
\end{arabiclist}
\end{proposition}
\begin{proof}
The equivalence of (1) and (2) follows because all countable non-scattered
linear orders are equimorphic and $\Q$ is strongly surjective. By the above
observation $L$ is strongly surjective if and only if $L \equiv_s \Q$, which
in turn is equivalent to $\Q \strong L$ because $L \strong \Q$ for every
countable $L$ by Proposition 16(1) in \cite{cacama}.

The equivalence of (2) and (3) follows from Proposition 17 in \cite{cacama}.
\end{proof}

\begin{definition}
Given a linear order $L$ without a maximum, the {\it cofinality} of $L$,
denoted $cof(L)$, is the smallest ordinal number $\alpha $ such that there
exists an increasing function $\alpha\to L$ unbounded above in $L$.

Similarly, for a linear order $L$ without a minimum, the \emph{coinitiality}
of $L$, denoted $coi(L)$, is the reverse $\alpha^{\star }$ of the smallest
ordinal $\alpha $ such that there exists an increasing function
$\alpha^{\star }\to L$ unbounded below in $L$. Equivalently,
$coi(L)=(cof(L^{\star }))^{\star }$.

Recall that $L$ is \emph{short} means that $\omega_1 \not\inj L$ and
${\omega_1}^\star \not\inj L$.
\end{definition}

Recall the following fact (\cite{cacama}, Fact 14(5)):

\begin{proposition}\label{cofcoi}
If $K$ and $L$ have no maximum and $K \strong L$, then $cof(K) = cof(L)$.
Similarly, if $K,L$ have no minimum and $K \strong L$, then
$coi(K)=coi(L)$.
\end{proposition}

\begin{proposition}\label{basics2}~
\begin{arabiclist}
\item If a strongly surjective order has a minimum, then it is a
    well-order. If it has a maximum, then it is the reverse of a
    well-order.
\item A strongly surjective linear order that is not an ordinal has
    coinitiality $\omega^{\star }$. Similarly, a strongly surjective linear
    order that is not the reverse of an ordinal has cofinality $\omega$.
\item Every strongly surjective linear order is short.
\item The cardinality of a strongly surjective linear order cannot exceed
    the continuum
\end{arabiclist}
\end{proposition}
\begin{proof}
(1) If $L$ is a strongly surjective order with a minimum and $K$ is a
non-empty subset of $L$, then $K$ must have a minimum, otherwise $K \strong
L$ would be impossible.
Similarly for the maximum.

(2) If $L$ is an ill-founded strongly surjective order, then $\omega^{\star
}\inj L $ and so $\omega^{\star }\strong L$. It also follows that $L$ does
not have a minimum. So $coi(L)=\omega^{\star }$ by Proposition \ref{cofcoi}.
Similarly for the cofinality.

(3) By Proposition \ref{cofcoi}, (1) and (2) any suborder of a strongly
surjective order $L$ must have either a maximum or cofinality $\omega$.
Therefore $\omega_1 \not\inj L$. Similarly ${\omega_1}^\star \not\inj L$.

(4) follows from (3) and a classical theorem of Urysohn's (\cite{Rosens1982}, Theorem
9.28) about short linear orders.
\end{proof}

It is useful to give a name to the orders satisfying the necessary conditions
for strong surjectivity given in the first two items of Proposition
\ref{basics2}.

\begin{definition}
A linear order $L$ is \emph{admissible} if the following conditions hold:
\begin{arabiclist}
\item $L$ has a miminum or it has coinitiality $\omega^{\star}$;
\item $L$ has a maximum or it has cofinality $\omega$.
\end{arabiclist}
\end{definition}

So an order is short if and only if it and all of its suborders are
admissible.


\subsection{Defining epimorphisms}
Given non-empty convex subsets $K_0$ and $K_1$ of $K$, say that $K_0\leq K_1$
when for all $x\in K_0$ and $y\in K_1$ we have $x\leq_Ky$; similarly define
$K_0<K_1$ if for all $x\in K_0$ and all $y\in K_1$ one has $x<_Ky$. Say that
$K_0$ and $K_1$ are \emph{adjacent} if $K_0\leq K_1$ and there is no $x\in K$
satisfying $K_0<\{x\}<K_1$. Say they are \emph{connected} when $K_0 \leq K_1$
but $K_0 \nless K_1$ (so that they share an element).

An epimorphism can be defined on a covering by convex sets.

\begin{definition}\label{def:nice}
We say that a family of non-empty convex sets $(K_i)_{i\in I}$ of an order
$K$ is \emph{nice} if and only if the index set $I$ is an interval of $\Z$,
the family $(K_i)_{i\in I}$ is unbounded above and below in $K$ and for all
$i\in I$, $K_i\leq K_{i+1}$ holds.

We say that $(K_i)_{i\in I}$ is a \emph{nice covering of $K$} if it is a
covering of $K$ by a nice family.
\end{definition}

\begin{lemma}[Definition by pieces]\label{definitionbypieces}
Suppose we have $(K_i)_{i\in I}$ a nice family of convex subsets of $K$ and
$(L_i)_{i\in I}$ a nice covering of $L$ satisfying that for any $i\in I$ when
$K_i$ and $K_{i+1}$ are not adjacent then $L_i$ has a maximum or $L_{i+1}$
has a minimum, and if $K_i$ and $K_{i+1}$ are connected, so are $L_i$ and
$L_{i+1}$.

If for all $i\in I$, $L_i\strong K_i$ holds, then $L\strong K$.
\end{lemma}

\begin{proof}
Take $\sigma_i\in\Epi(L_i,K_i)$ and whenever $K_i$ and $K_{i+1}$ are not
adjacent let $l_i$ be the maximum of $L_i$ if it exists, or the minimum of
$L_{i+1}$ otherwise. Define then the map $\sigma:K\rao L$ as follows.
\[
\sigma(x)=\begin{cases}
\sigma_i(x)&\mbox{if }x\in K_i\mbox{ for some }i\in I\\
l_i& \mbox{if }K_i<\{x\}< K_{i+1}\mbox{ for some }i\in I
\end{cases}
\]
We defined $\sigma$ on every $K_i$ and on the convex sets between $K_i$ and
$K_{i+1}$, so on all of $K$. Let us first check that it is well-defined.
Suppose $x$ is in $K_i\cap K_{i+1}$ for some $i\in I$. Then, since $K_i\leq
K_{i+1}$, we have $K_i\cap K_{i+1}=\{x\}$, so that $x=\max K_i=\min K_{i+1}$
connects the two intervals. The hypothesis gives that $\max L_i=\min
L_{i+1}$, and as the maps $\sigma_i$ and $\sigma_{i+1}$ are epimorphisms we
have
\[
\sigma_i(x)=\sigma_i(\max K_i)=\max L_i=\min L_{i+1}=\sigma_{i+1}(\min K_{i+1})=\sigma_{i+1}(x),
\]
so $\sigma$ is indeed well-defined. Since the maps $\sigma_i$ are
epimorphisms and the sets $L_i$ form a nice covering of $L$, we finally
have $\sigma\in\Epi(L,K)$.
\end{proof}

In the above proof, we say that $\sigma$ is \emph{defined by pieces}.
Some specific operations come in
handy to define epimorphisms by pieces.

\begin{definition}
Given $K,L$ linear orders, $\sigma\in\Epi(L,K)$ and $l\in L$, we denote
$\sigma^l$ the following epimorphism:
\begin{align*}
\sigma^l:K&\lgrao ({\lao},l]_L\\
k&\lgto\begin{cases}
\sigma(k)&\mbox{if }\sigma(k)\leq_L l\\
l&\mbox{otherwise.}
\end{cases}
\end{align*}
Similarly we define
\begin{align*}
\sigma_l:K&\lgrao[l,{\rao})_L\\
k&\lgto\begin{cases}
\sigma(k)&\mbox{if }\sigma(k)\geq_L l\\
l&\mbox{otherwise.}
\end{cases}
\end{align*}
Given $l\leq_Ll'$, we also define
\begin{align*}
\sigma_l^{l'}:K & \lgrao [l,l']_L \\
k & \lgto
\begin{cases}
l & \mbox{if } \sigma (k)<l \\
\sigma (k) & \mbox{if } l\leq\sigma (k)\leq l' \\
l' & \mbox{if } l'<\sigma (k).
\end{cases}
\end{align*}
\end{definition}

\begin{proposition}[Family mash]\label{familymash}
Given linear orders $K$ and $L$, with $L$ admissible, if there is a nice
family \((K_i)_{i\in I}\) of \(K\) such that $L \strong K_i$ holds for all
$i\in I$, then we have $L \strong K$.
\end{proposition}
\begin{proof}
We may assume that $L$ is not a singleton, and so no $K_i$ is a singleton.
Moreover, we can suppose that $I$ has more than one element, otherwise if
$I=\{i\}$ then $K_i=K$ (because $K_i$ is convex and unbounded) and we are
done.

For each $i \in I$ fix $\sigma_i \in \Epi(L,K_i)$. We want to define $\sigma
\in \Epi(L,K)$ by pieces.

Suppose first that $I$ is finite, and let $i$ and $j$ be its minimum and
maximum. Notice that $(K_i, K_j)$ is nice. On the $L$
side, take any $l\in L$ and consider the nice connected covering
$(({\lao},l]_L, [l,{\rao})_L)$. Then we can use Lemma
\ref{definitionbypieces}, with $(\sigma_i)^l$ and $(\sigma_j)_l$ witnessing
$({\lao},l]_L \strong K_i$ and $[l,{\rao})_L \strong K_j$ respectively, to
define $\sigma$.

From now on, we suppose that $I$ is infinite. There are four cases.

\begin{arabiclist}
\item When $L$ has a minimum $l_0$ and a maximum $l_1$, choose $i
    \in I$ different from the minimum and maximum of $I$ (at most one of
    the extrema exists), and let $K^- = \{ k\in K\mid \{k\} < K_i\}$ and
    $K^+ = \{ k\in K\mid K_i < \{k\}\}$. The nice covering $(K^-, K_i,
    K^+)$ of $K$ and the connected covering $(\{l_0\}, L, \{l_1\})$ of $L$
    allow the definition by pieces of $\sigma$.

\item When $L$ has a minimum $\hat{l}$ and no maximum, we need to
    distinguish two subcases.

    If $I$ has a maximum $i$ we consider the nice covering $(K\setminus K_i,
    K_i)$ and the connected covering $(\{\hat{l}\}, L)$.

If instead $I$ has no maximum, by admissibility of $L$ let $\{l_n\}_{n \in
\N}$ be strictly increasing and cofinal in $L$. Fix $i \in
I$ and let $K_i'=
K_i \cup \{ k \in K \mid \{k\} < K_i\}$. Now consider the nice family
$(K_i', K_{i+1}, K_{i+2}, \ldots)$ and the nice connected covering
\[
([\hat{l},l_i]_L, [l_i, l_{i+1}]_L, [l_{i+1}, l_{i+2}]_L, \ldots).
\]
Using $(\sigma_i)^{l_i}$ and $(\sigma_j)^{l_j}_{l_{j-1}}$ for $j>i$ we
again get the definition by pieces of $\sigma$.

\item When $L$ has a maximum and no
    minimum, we can just mirror the previous case.

\item When $L$ has no extrema, we look for a connected nice
    covering $(L_i)_{i\in I}$ to match  $(K_i)_{i\in I}$. Take
    $(l_i)_{i\in\Z}$ strictly increasing, coinitial and cofinal in $L$. If
    $I=\Z$ then take $([l_i,l_{i+1}]_L)_{i\in I}$. If $I$ has a minimum
    $j$, take $L_j=({\lao}, l_j]_L$ and $([l_i, l_{i+1}]_L)_{i\geq j}$. If
    $I$ has a maximum $j$, take $([l_i, l_{i+1}]_L)_{i<j}$ and $L_j =
    [l_j,{\rao})_L$. In any case, we can use the
    appropriate $(\sigma_i)_{l_i}^{l_{i+1}}$, $(\sigma_j)_{l_j}$, and
    $(\sigma_j)^{l_j}$ to define $\sigma$ by pieces. 
\end{arabiclist}
\end{proof}

In the above proof we say that we \emph{mash} the family $(K_i)_{i\in I}$
onto $L$.

\begin{corollary}\label{cor:lessthanprod}
For $K$ and $L$ admissible, we have $L \strong LK$.
\end{corollary}
\begin{proof}
By admissibility of $K$, take $(k_i)_{i\in I}$ increasing, coinitial and
cofinal in $K$ for some $I$ which is an interval in $\Z$. Mash the nice
family $(L\times\{k_i\})_{i\in I}$ of $LK$ onto $L$ using Proposition
\ref{familymash}.
\end{proof}


\subsection{Operations on strongly surjective orders} \label{operations}

In general, the sum of strongly surjective orders is not strongly surjective:
consider for example a countable ordinal whose Cantor normal form has two summands and
use Theorem \ref{ssordinal}. We now show instead that the product of two
strongly surjective orders is still strongly surjective, and that the
left-quotient of a strongly surjective order by a scattered order is also
strongly surjective. First note the following.

\begin{proposition} \label{ssrestricted}
Let $I$ be any order and for each $i\in I$, let $L_i$ be a strongly
surjective order. Then $L=\sum_{i\in I}L_i$ is strongly surjective if and
only if for every non-empty $J\subseteq I$ there is an epimorphism from $L$
onto $L_J=\sum_{j\in J}L_j$.
\end{proposition}

\begin{proof}
If $L$ is strongly surjective, then it must admit an epimorphism onto its
suborder $L_J$ for any non-empty $J\subseteq I$.

Conversely, suppose there is an epimorphism $\psi_J:L\to L_J$ for any
non-empty $J\subseteq I$. Let $K$ be a suborder of $L$ and let $J$ be the set
of indices $j$ such that $K$ intersects
$L_j$ in a non-empty set $K_j$, so that $K=\sum_{j\in J}K_j$. Since each
$L_j$ is strongly surjective, let $\varphi_j:L_j\to K_j$ be an epimorphism.
These induce an epimorphism $\varphi :L_J\to K$. Then $\varphi \circ \psi_J:
L \to K$ is an epimorphism.
\end{proof}

This yields the following simple examples of strongly surjective orders.

\begin{example} \label{ordinalstarplusordinal}
Let $\gamma ,\delta $ be countable ordinals and $n,m>0$. Then
$(\omega^{\gamma }n)^\star,\omega^{\delta }m$ and $(\omega^{\gamma
}n)^\star+\omega^{\delta }m$ are strongly surjective.
\end{example}
\begin{proof}
The fact that $(\omega^{\gamma }n)^\star,\omega^{\delta }m$ are strongly
surjective is a consequence of Theorem \ref{ssordinal} and Proposition
\ref{basics}(1). So by Proposition \ref{ssrestricted} it is enough to show
$(\omega^{\gamma }n)^\star \strong (\omega^{\gamma }n)^\star+\omega^{\delta
}m$ and $\omega^{\delta }m \strong (\omega^{\gamma }n)^\star+\omega^{\delta
}m$, which can be done by a definition by pieces using the presence of an
extremum in the range.
\end{proof}

\begin{corollary} \label{multiple}
If $L$ and $M$ are strongly surjective, then $LM$ is strongly surjective. In
particular $L^n$ is strongly surjective for all $n\in\N$.
\end{corollary}
\begin{proof}
By Proposition \ref{ssrestricted}, it is enough to show $LK \strong LM$ for
any suborder $K$ of $M$. Let $\varphi: M \to K$ be an epimorphism and for $k
\in K$ let $M_k= \varphi^{-1} (\{k\})$. As $M$ is strongly surjective, each
$M_k$ must be admissible. Since $L$ is also admissible by Proposition
\ref{basics2}, Corollary \ref{cor:lessthanprod} implies that there is an
epimorphism $\varphi_k:L M_k\to L$ for every $k \in K$. Gluing together these
epimorphisms yields $LK \strong LM$.
\end{proof}

Strongly surjective orders are not closed under infinite products (ordered
lexicographically), as we will show in Section \ref{unctbless}.

\begin{lemma}\label{manytofew}
If $L$ is a scattered linear order and $0<n<m$, then $Lm \not \inj
Ln$.
\end{lemma}
\begin{proof}
We first show the special case $L2 \not \inj L$, which is actually
Lemma 1.17 of \cite{Mon06}. Notice that if $L2 \inj L$ then an easy induction
shows that $Ln \inj L$ for any $n$. We show that under this hypothesis $L$ is
not scattered. To this end we recursively define for every $s \in
2^{< \N }$ a subset $L_s$ of $L$ which is isomorphic to $L$ and a point
$x_s \in L$. Start with $L_{\es} =L$. Assuming that $L_s \simeq L$ and $L_s
\subseteq L$, since $L3 \inj L_s$, pick a point $x_s$ in the middle copy of
$L$ embedded in $L_s$ and let $L_{s0}$ and $L_{s1}$ be the left and right
copies of $L$ embedded in $L_s$. Then $\{x_s\}_{s\in 2^{< \N }}$ is a
dense suborder of $L$, and so $L$ is not scattered.

Now assume $0<n<m$ and $Lm \inj Ln$. Again inductively
one can show that $L(m+k(m-n)) \inj Ln$ for all $k$. If $k$ is large enough
we have $m+k(m-n) \geq 2n$ and hence $Ln2 \inj Ln$, which by the above
implies that $Ln$ is not scattered. Thus $L$ is not scattered.
\end{proof}

\begin{proposition} \label{quotient}
If $L$ is scattered and $LK$ is strongly surjective, then $K$ is strongly
surjective.
\end{proposition}
\begin{proof}
Let $J \inj K$ and fix an epimorphism $\varphi :LK\to LJ$. Define the
relation $R\subseteq K\times J$ by letting $kRj\Leftrightarrow\varphi
(L\times\{ k\} )\cap (L\times\{ j\} )\neq\es $. If $k \in K$ we denote by
$R_k$ the vertical section $\{j \in J \mid k R j\}$. Similarly, for $j \in
J$, $R^j$ is the horizontal section $\{k \in K \mid k R j\}$. Notice that:
\begin{itemize}
\item[{\bf -}] all sections are non-empty (i.e. the
    domain of $R$ is $K$ and its range is $J$);
\item[{\bf -}] all sections are convex subsets of
    the respective linear order;
\item[{\bf -}] $|R_k| \leq 3$ for each $k\in K$ (by
    Lemma \ref{manytofew}).
\end{itemize}
To define an epimorphism $\psi :K\to J$, we only need to define a surjection
$\psi$ that satisfies $kR\psi(k)$ for all $k\in K$. Given $k\in K$ we
distinguish several cases.
\begin{itemize}
\item[(a)] If $|R_k| =1$ and $j$ is the unique element of $R_k$ set $\psi (k)=j$.
\item[(b)] If there is $j$ (necessarily unique, by Lemma \ref{manytofew})
    such that $L\times\{ j\}\subseteq\varphi (L\times\{ k\} )$, then set
    $\psi (k)=j$; note that if $|R_k| = 3$ then $k$
    satisfies this case.
\item[(c)] So it remains to define $\psi $ on the set $H$ of those $k$
    such that $|R_k| = 2$, but do not fall in case
    (b).
Consider $I$ a maximal $\leq_K$-convex subset of $H$ that is embeddable in
    $\Z$: $I$ is contained in a $\mathbf{c}_F$-condensation class $C$ of
    $K$; see Section 4.2 in \cite{Rosens1982}. So $I$ has order type finite,
    $\omega$, $\omega^{\star}$ or $\omega^{\star} + \omega$. We need to
    define $\psi $ on each such $I$.
\item[(c1)] Suppose first that $I$ has $n$ elements $k_0<\ldots <k_{n-1}$.
    Consequently, $\bigcup_{r=0}^{n-1}R_{k_r}$ consists of $n+1$
    consecutive points of $J$, say $j_0<\ldots <j_n$, so that $R_{k_r}=\{
    j_r,j_{r+1}\} $.
\item[(c1a)] If $k_0=\min C$, set $\psi (k_r)=j_{r+1}$ for $0\leq r\leq
    n-1$. Notice that in this case, since $j_0$ does not witness that case
    (b) applies to $k_0$, $R^{j_0}$ consists of $k_0$ and an infinite convex set with supremum
    $k_0$: then for all but at most one $k \in R^{j_0} \setminus \{k_0\}$
    we have $\psi(k) = j_0$ by case (a).
\item[(c1b)] If $k_0 \neq \min C$ but $k_{n-1}=\max C$, then let $\psi
    (k_r)=j_r$ for $0\leq r\leq n-1$. An argument
    similar to the one used in case (c1a) shows that in this case $\psi(k)
    = j_n$ for some $k$.
\item[(c1c)] If neither $k_0$ is the first element of $C$, nor $k_{n-1}$ is
    the last element of $C$, let $k'$ be the immediate precedessor of $k_0$
    and $k''$ be the immediate successor of $k_{n-1}$
    in $K$. Thus $k'Rj_0$, $k''Rj_n$
(because neither $k_0$ nor $k_{n-1}$ satisfy the condition of case (b)) and
    both $\psi (k')$ and $\psi (k'')$ have been defined according to cases
    (a) or (b). Notice that $\psi (k')$ and $\psi (k'')$ cannot have both
    been defined according to clause (b), with values different from $j_0,j_n$, respectively, as in this case one would have
    $L(n+3) \inj L(n+2)$, contradicting Lemma \ref{manytofew}. This implies
    that either $\psi (k')=j_0$ or $\psi (k'')=j_n$. If $\psi (k')=j_0$,
    let $\psi (k_r)=j_{r+1}$; otherwise, let $\psi (k_r)=j_r$.
\item[(c2)] If $I$ has order type $\omega $, $\omega^{\star }$ or
    $\omega^{\star} + \omega $, then $\bigcup_{k\in I}R_k$ has the same
    order type and we can define $\psi $ on $I$ as any order preserving
    surjection onto $\bigcup_{k\in I}R_k$.
\end{itemize}
By construction, $\psi :K\to J$ is order preserving and surjective.
\end{proof}


\section{Bounding the complexity of $\StS$ from below}\label{lower bounds}

The closure properties of Subsection \ref{operations} allow to build several
examples of strongly surjective linear orders. We present here other kinds of
examples allowing to obtain some hardness results.

First we make our formal setting precise. We call $\lin$ the
subset of $2^{ \N \times \N }$ consisting of all linear orders $\leq_K$ on a
subset $K=\dom(\leq_K)$ of $ \N $. By definition it is a Polish subspace of
$2^{ \N \times \N }$. To avoid heavy notations, when there is no possible
confusion we just write $K$ for the pair $(K,{\leq_K})$.

When we work with elements of $\lin$ we fix recursive copies of $\N$ and
$\Q$, denoted respectively by $\omega$ and $\eta$. Moreover we assume a fixed
way of implementing sums (finite or infinite) and products as recursive (and
hence continuous) operations which produce new elements of $\lin$.

\begin{remark}
In the literature most often people work with $\mathsf{LO}$, the space of all
total orders on the domain $ \N $. The downside of $\mathsf{LO}$ is the
absence of finite orders, which we need for the main result in Section
\ref{analysis}. That is why we deal with $\lin$. However, for the
classification results on strongly surjective orders the two settings are
equivalent. Indeed, denote by $\mathsf{Fin}$ the $\Sigma^0_2$ set of finite
orders in $\lin$, and notice that there are continuous functions $\mathsf{LO}\to \lin \setminus \mathsf{Fin}$ and $\lin \setminus
\mathsf{Fin} \to \mathsf{LO}$ that preserve order types. If $\Gamma$ is a pointclass that includes $
\boldsymbol \Sigma^0_2$ and is closed under finite unions and continuous
preimages, and the set of strongly surjective orders in $\mathsf{LO}$ belongs
to $\Gamma$, then $\StS \setminus \mathsf{Fin} \in \Gamma$, so $\StS
\in\Gamma$. Conversely, if $\StS \in\Gamma$, one has that the strongly
surjective orders as a subset of $\mathsf{LO}$ are in $\Gamma$ as well.
\end{remark}

\subsection{Basic hardness}\label{basic}
Let $\scat$ and $\WO$ be the subsets of $\lin$
consisting of the scattered countable linear orders and of the countable
well-orders.
It is well-known that both $\scat$ and
$\WO$ are $\Pi^1_1$ and $\boldsymbol \Pi^1_1$-complete.

\begin{proposition}\label{basichardness1}
The sets $\StS$, $\StS\cap\scat$ and $\StS\cap\WO$ are $ \boldsymbol
\Pi^1_1$-hard.
\end{proposition}
\begin{proof}
Let $g:\lin\to\lin $ be defined by $g(L) = (1+L)
\omega$. Using Proposition \ref{basics2} and Theorem \ref{ssordinal}, as
$g(L)$ has a minimal element for any $L$, we have $g(L)\in\StS$ if and only
if $g(L)\in\WO$. Since $g(L)\in\WO$ if and only if $L\in\WO$, we have that
$g$ reduces $\WO$ to $\StS$, to $\StS\cap\scat$, and to $\StS\cap\WO$ as
well.
\end{proof}

We now consider the set of countable strongly surjective
linear orders that are non-scattered.

\begin{proposition}\label{basichardness2}
The set $\StS$ is $\mathbf{\Sigma}^1_1$-hard, and the set of non-scattered
strongly surjective orders is $\Sigma^1_1$ and $\mathbf{\Sigma}^1_1$-complete.
\end{proposition}
\begin{proof}
Let $f: \lin \rao \lin$ be defined by $f(L) = \eta + L\omega$. As $f(L)$ is
non-scattered for all $L$, we have $f(L) \in \StS$ if and only if it has no
scattered initial nor final segments by Proposition \ref{nonscat}. But $f(L)$
never has a scattered initial segment, and it has a scattered final segment
if and only if $L$ itself is scattered. So finally $f$ reduces
$\lin\setminus\scat$ to $\StS$, and even to the set of non-scattered strongly
surjective countable linear orders, which are consequently
$\mathbf{\Sigma}^1_1$-hard.

The fact that $\StS \setminus \scat$ is $\Sigma^1_1$ follows from the
characterization of Proposition \ref{nonscat}.(ii).
\end{proof}

\subsection{Powers of $\Z$}
The main new ingredient needed for the lower bound is a general version of
the exponentiation with base $\Z$. There are two definitions of
$\Z^\alpha$ for $\alpha$ an ordinal number. The first one is by ordinal
induction (see \cite{Rosens1982}, Definition 5.34), while the second
(\cite{Rosens1982}, Definition 5.35) is a direct set theoretic definition,
and it can actually be used as a definition of $\Z^K$ for $K$ any linear
order. As pointed out in \cite{Rosens1982}, Exercise 5.36(1), the two
definitions coincide when $K$ is a well-order.

We first recall the definition by ordinal induction.

\begin{definition}~
\begin{arabiclist}
\item $\Z^0=1$,
\item $\Z^{\alpha+1} = \Z^\alpha
    \omega^\star + \Z^\alpha + \Z^\alpha \omega$,
\item $\Z^\alpha = \big(\sum_{\beta<\alpha} \Z^\beta\omega\big)^\star
    + 1 + \sum_{\beta<\alpha} \Z^\beta\omega$ if $\alpha$ is a limit
    ordinal.
\end{arabiclist}
\end{definition}

The following
equalities will be useful.

\begin{proposition}\label{eq:powersofZ}
For any $\alpha$ and $\beta<\alpha$, we have
\[
\Z^{\alpha} =
\big(\sum_{\gamma <\alpha }\Z^{\gamma }\omega \big)^{\star }+1+\sum_{\gamma <\alpha }\Z^{\gamma }\omega
=\big(\sum_{\beta\leq\gamma <\alpha }\Z^{\gamma }\omega\big)^\star+\sum_{\beta\leq\gamma <\alpha }\Z^{\gamma }\omega.
\]
\end{proposition}
\begin{proof}
To prove the first equality we argue by induction on $\alpha$. The cases
$\alpha=0$ and $\alpha$ limit are immediate from the definition. For the
successor case we have
\begin{align*}
\Z^{\alpha+1}& = \Z^{\alpha} \omega^{\star} + \Z^{\alpha} + \Z^{\alpha} \omega \\
&= \Z^{\alpha} \omega^{\star} + \big(\sum_{\gamma <\alpha }\Z^{\gamma}\omega
\big)^{\star }+1+\sum_{\gamma <\alpha }\Z^{\gamma }\omega +
\Z^{\alpha}\omega \\
&= \big(\sum_{\gamma <\alpha+1}\Z^{\gamma }\omega \big)^{\star }+1+\sum_{\gamma <\alpha+1}\Z^{\gamma }\omega,
\end{align*}
where in the central step we use the induction hypothesis.

The second equality can be proved applying the first one to $\beta$, using
$(\Z^{\beta} \omega)^{\star} + \Z^{\beta} \omega = \Z^{\beta}\Z =
\Z^{\beta +1}$.
\end{proof}

\begin{proposition}\label{prop:powersofZ}
For any countable ordinal $\alpha $ and natural number $m>0$, the order
$\Z^{\alpha }m$ is strongly surjective.
\end{proposition}
\begin{proof}
Since finite linear orders are trivially strongly surjective, by Corollary
\ref{multiple} it suffices to show that each $\Z^{\alpha}$ is strongly
surjective. Proceed by induction on $\alpha$. When $\alpha=0$ we get the
singleton linear order. Notice that $\Z$ is strongly surjective by Example
\ref{ordinalstarplusordinal}, so that Corollary \ref{multiple} handles the
successor step because $\Z^{\alpha+1} = \Z^\alpha \Z$.

Suppose now that $\delta $ is limit and that $ \Z^{\gamma }$ is strongly
surjective for all $\gamma<\delta $. By Corollary \ref{multiple}, so are
$\Z^{\gamma }\omega$ and $\Z^{\gamma} \omega^\star$. Recall that, by
Proposition \ref{eq:powersofZ}, $\Z^{\delta}$ can be written as a sum over
the index set $I = \delta^\star + \delta $: $ \Z^{\delta }=(\sum_{\gamma
<\delta } \Z^{\gamma }\omega )^{\star}+\sum_{\gamma <\delta } \Z^{\gamma
}\omega $.

First we show that if $1\leq\beta_0<\beta_1\leq\delta $ and $H$ is a
non-empty subset of $\delta $ with $\sup H<\beta_0$,
then

\begin{equation}\label{equ:1}
\sum_{\gamma\in H}\Z^\gamma\omega \strong \sum_{\beta_0\leq\gamma<\beta_1} \Z^\gamma\omega.
\end{equation}
If $\beta_1=\rho +1$ is a successor ordinal, $(\Z^{\beta_0} \omega,
\Z^{\rho } \omega )$ is a nice family in $\sum_{\beta_0 \leq
\gamma<\beta_1}\Z^\gamma \omega$ and we can mash
$\sum_{\beta_0\leq\gamma<\beta_1} \Z^\gamma\omega$ onto $\sum_{\gamma\in
H} \Z^\gamma\omega $, since $\sum_{\gamma\in H} \Z^\gamma\omega \inj
\Z^{\beta_0} \inj \Z^{\beta_0}\omega \inj \Z^{\rho}\omega $. If
$\beta_1$ is limit, let $\rho_n$ be an increasing cofinal sequence in
$\beta_1$, with $\rho_0=\beta_0$: we can mash the nice family
$(\Z^{\rho_n}\omega)_{n\in \N }$ onto $\sum_{\gamma\in
H}\Z^\gamma\omega $. So, in either case we get (\ref{equ:1}).

Take now a non-empty subset $J$ of $I$: it determines
two subsets $J^-,J^+\subseteq\delta $ -- one of them possibly empty -- and a
suborder $K=(\sum_{\gamma\in J^-}\Z^{\gamma }\omega
)^\star+\sum_{\gamma\in J^+}\Z^{\gamma }\omega $. We want to show that $K
\strong \Z^{\delta }$, so that we can conclude the proof by applying
Proposition \ref{ssrestricted}. Set $K^-=(\sum_{\gamma\in J^-}\Z^{\gamma
}\omega )^{\star }$ and $K^+=\sum_{\gamma\in J^+}\Z^{\gamma }\omega$.

First notice that we may suppose that both $K^-$ and $K^+$ are non-empty. In
fact if, for example, $K^-=\es $, letting $\alpha =\min J^+$, we have
\[
K=K^+=\Z^{\alpha }+K^+ \strong (\Z^{\alpha }\omega )^{\star }+K^+,
\]
so that an epimorphism from $\Z^\delta$ onto the rightmost part gives by
composition an epimorphism onto $K$.

Similarly, we may assume that both $J^-$ and $J^+$ are unbounded in $\delta
$. Indeed, if $\alpha <\delta $ is an upper bound for, say, $J^+$, we have
\[
K \strong K^-+\Z^{\alpha +1} \strong K^-+\Z^{\alpha +1}\omega \strong K^-+\sum_{\alpha+2\leq\gamma <\delta }\Z^{\gamma }\omega,
\]
where in the first inequality we used the induction
hypothesis, and in the last one (\ref{equ:1}).

Since $\Z^{\delta} = (\sum_{\beta \leq \gamma < \delta} \Z^{\gamma}
\omega)^{\star} + \sum_{\beta \leq \gamma < \delta} \Z^{\gamma} \omega$ for
any $\beta < \delta$, it is enough to show that both $K^- \strong
(\sum_{\beta \leq \gamma < \delta} \Z^{\gamma} \omega)^{\star}$ and $K^+
\strong \sum_{\beta \leq \gamma < \delta} \Z^{\gamma} \omega$ hold for some
$\beta $. To prove, for instance, the latter, let $\{\alpha_n\}_{n\in \N }$
be increasing and cofinal in $J^+$, with $\alpha_0=\min J^+$, and let
$\{\beta_n\}_{n\in \N }$ be increasing and cofinal in $\delta $, with
$\alpha_{n+1}<\beta_n$ for all $n\in \N$. Then, by (\ref{equ:1}), there exist
epimorphisms
\[
\sum_{\beta_n\leq\gamma <\beta_{n+1}}\Z^{\gamma }\omega \to
\sum_{\gamma\in J^+,\alpha_n\leq\gamma< \alpha_{n+1}}\Z^{\gamma }\omega
\]
Gluing them together, one obtains an epimorphism from
$\sum_{\beta_0\leq\gamma <\delta }\Z^{\gamma }\omega$ onto $K^+$.
\end{proof}

Here is the set-theoretic definition of exponentiation with base $\Z$.

\begin{definition}
Let $K$ be a linear order. For any map $\varphi:K \to \Z$, $\Supp(\varphi)$
stands for the \emph{support} of $\varphi$, that is $\Supp(\varphi) = \{k \in
K\mid \varphi(k) \neq 0\}$. The $K$-power of $\Z$, denoted by $\Z^K$,
is the following order on $\{\varphi: K \to \Z \mid \Supp(\varphi) \text{ is
finite}\}$. If $\varphi, \psi: K \to \Z$ are maps with finite support let
$\varphi \leq_{\Z^K} \psi$ if and only if $\varphi = \psi$ or
$\varphi(k_0) <_\Z \psi(k_0)$ where $k_0 = \max \{ k \in \Supp(\varphi) \cup
\Supp(\psi) \mid \varphi(k) \neq \psi(k)\}$.
\end{definition}

We now show that if $K$ is countable but not a well-order then $\Z^K \equiv_s
\Q$, and hence $\Z^K$ is strongly surjective by Proposition \ref{nonscat}.

\begin{lemma}For any linear orders $K$ and $L$ we have:
\begin{arabiclist}
\item $\Z^{K+L} \simeq \Z^K \Z^L$
\item if $K$ is countable and with no minimum then $\Z^K \simeq \Q$.
\end{arabiclist}
\end{lemma}

\begin{proof}
(1) The bijection $\Z^{K+L}\to\Z^K\Z^L$,
$\varphi\mapsto(\varphi\rest{K},\varphi\rest{L})$ is an isomorphism.

(2) Take $K$ countable with no minimum, and suppose $\varphi<_{\Z^K}\psi$
holds for some $\varphi$ and $\psi$ in $\Z^K$. As
$K$ has no minimum pick $k_0 \in K$ that is strictly below every element of
$\Supp(\varphi)\cup\Supp(\psi)$. Define $\varphi^-$, $\theta$ and $\psi^+$,
all in $\Z^K$, as follows.
\begin{align*}
\varphi^-(x)=&\begin{cases}
-1&\mbox{if }x=k_0\\
\varphi(x)&\mbox{otherwise,}
\end{cases}\\
\theta(x)=&\begin{cases}
1&\mbox{if }x=k_0\\
\varphi(x)&\mbox{otherwise,}
\end{cases}\\
\psi^+(x)=&\begin{cases}
1&\mbox{if }x=k_0\\
\psi(x)&\mbox{otherwise.}
\end{cases}
\end{align*}
We have
\[\varphi^-<_{\Z^K}\varphi <_{\Z^K}\theta <_{\Z^K}\psi <_{\Z^K}\psi^+,\]
so $\Z^K$ is dense, countable, without extrema, giving $\Z^K \simeq \Q$.
\end{proof}

\begin{proposition}\label{prop:nonordinalpowerofZ}
If $K$ is countable and not a well order then there is a countable ordinal
$\alpha$ such that
\[
\Z^K\simeq\Z^\alpha\Q.
\]
Hence $\Z^K \equiv_s \Q$.
\end{proposition}
\begin{proof}
To obtain $\Z^K\simeq\Z^\alpha\Q$ it suffices to use the previous lemma with
the decomposition $K=\alpha+K'$ for $\alpha$ ordinal and $K'$ without a
minimum. Since $\Z^K$ is written as a $\Q$ sum we have $\Q \strong \Z^K$,
which yields $\Z^K \equiv_s \Q$ because $K$, and hence $\Z^K$, is countable.
\end{proof}

\subsection{The lower bound}
We now prove that $\StS$ is hard for the class
$\check{\mathrm{D}}_2(\mathbf{\Pi}^1_1)$ of all sets that are the union of an
analytic and a coanalytic set.

As we did for sums and products, we want to realize exponentiation with base
$\Z$ as an operation on $\lin$. Since in this case the details are less
straightforward, we provide them:

\begin{proposition}\label{Psiexp}
There is a continuous (even recursive) function $\lin\to\lin$ mapping
any $K$ to an order $(\zeta^K,\leq_{\zeta^K})$ isomorphic to $\Z^K$.
\end{proposition}
\begin{proof}
Fix $n\mapsto s_n$ and $n\mapsto((n)_0,(n)_1)$ recursive enumerations of $
\N^{< \N }$ and $ \N^2$ respectively, as well as a recursive order
$\prec_{\zeta }$, whose domain is the whole $\N$, that is isomorphic to the
strict part of $\leq_\Z$. For any $K \in \lin$ we define $(\zeta^K,
\leq_{\zeta^K}) \in \lin$.

First, the domain $\zeta^K$ is the set of codes for pairs of sequences of the
same length, the first with values in $K$, the second in $\N \setminus
\{0\}$. This simulates the finite support. For convenience we require the
sequences with values in $K$ to be $\leq_K$-decreasing. Writing $\lh(s)$ for
the length of a sequence $s$:
\begin{multline*}
\zeta^K=\Big\{n\in \N \mid\lh(s_{(n)_0})=\lh(s_{(n)_1})\wedge
\forall i< \lh (s_{(n)_0})\,
s_{(n)_1}(i)\neq0 \\ \wedge
\forall i<j<\lh(s_{(n)_0})\ s_{(n)_0}(j)<_Ks_{(n)_0}(i)\Big\}.
\end{multline*}
We now  compare two codes of pairs of sequences on the first value on which
they differ. We have to be careful because if the value differs on the first
sequence, it means that the two sequences do not have the same support. Also,
one sequence could extend the other. Formally, given $n,m \in \zeta^K$ we
have $n \leq_{\zeta^K} m$ if and only if
\begin{itemize}
\item[-] $n=m$;

or
\item[-] for all $i<\min\{\lh(s_{(n)_0}),\lh(s_{(m)_0})\}$ we have
\[\forall\varepsilon\in\{ 0,1\}\big(s_{(n)_\varepsilon}(i)=s_{(m)_\varepsilon}(i)\big),\]
and
\begin{multline*}
\big(\lh(s_{(n)_0})<\lh(s_{(m)_0})\wedge 0\prec_{\zeta }s_{(m)_1}(\lh(s_{(n)_0}))\big)\vee \\
\big(\lh(s_{(m)_0})<\lh(s_{(n)_0})\wedge s_{(n)_1}(\lh(s_{(m)_0}))\prec_{\zeta }0\big) ;
\end{multline*}

or
\item[-] there exists $i<\min\{\lh(s_{(n)_0}),\lh(s_{(m)_0})\}$ satisfying
\[\forall j<i~\forall\varepsilon\in\{ 0,1\}\big(s_{(n)_\varepsilon}(j)=s_{(m)_\varepsilon}(j)\big),\]
and
\begin{multline*}
\big(s_{(n)_0}(i)<_Ks_{(m)_0}(i)\wedge0\prec_{\zeta }s_{(m)_1}(i)\big)\lor\\
\big(s_{(m)_0}(i)<_Ks_{(n)_0}(i)\wedge s_{(n)_1}(i)\prec_{\zeta }0\big)\lor\\
\big(s_{(n)_0}(i)=s_{(m)_0}(i)\wedge s_{(n)_1}(i)\prec_{\zeta }s_{(m)_1}(i)\big).
\end{multline*}
\end{itemize}
This is, given $K$, a recursive encoding of an order $\zeta^K$ isomorphic to $ \Z^K$.
\end{proof}

\begin{theorem} \label{main}
The set $\StS$ is hard for the class $\check{\mathrm{D}}_2(\mathbf{\Pi}^1_1)$.
\end{theorem}

\begin{proof}
First observe that the set
\[
A=\{(K,L) \in \lin \times \lin \mid K \notin \WO \lor L \in \WO\}
\]
is $\check{\mathrm{D}}_2(\mathbf{\Pi}^1_1)$-complete: if $B=B_0\cup B_1$, with $f_0$ reducing $B_0$ to $\lin\setminus\WO$
and $f_1$ reducing $B_1$ to $\WO$, then $(f_0,f_1)$ reduces $B$ to $A$.

We now prove that $\StS$ continuously
reduces $A$. To this end we use the continuous map $f$ defined by
\[
f(K,L) = \zeta^K (1+L)\omega.
\]

If $K \notin \WO$ then by Proposition \ref{prop:nonordinalpowerofZ} $\Q
\strong \zeta^K$. Then $f(K,L)$ has no initial or final segment which is
scattered and by Proposition \ref{nonscat} we have $f(K,L) \in \StS$. If $K
\in \WO$ then $\zeta^K$ is isomorphic to an ordinal power of $\zeta$, which
is scattered and strongly surjective by Proposition \ref{prop:powersofZ}. In
that case, using both Proposition \ref{quotient} and Corollary
\ref{multiple}, $f(K,L) \in \StS$ if and only if $(1+L)\omega \in \StS $.
Finally, since $(1+L) \omega$ has a minimum, Proposition \ref{basics2}(1) and
Theorem \ref{ssordinal} tell us that $(1+L)\omega \in \StS$ if and only if $L
\in \WO$, which concludes the proof.
\end{proof}


\section{Bounding the complexity of $\StS$ from above} \label{analysis}

Given a set $A\neq\es$, the spaces $A^{ \N }$ and $A^{ \N \times \N }$ are
endowed with the product topology of the discrete topology on $A$.

Given $K, L$ in $\lin$, $\Emb(L,K)$ is a closed subspace of $K^L$, so it is
closed in $( \N \cup\{ *\} )^{ \N }$ as well, where $*$ is a new symbol (we
map the elements of $ \N \setminus L$ to $*$).
Similarly, the space $\Epi(L,K)$ is
$\mathbf{\Pi}^0_2$ in $L^K$ and hence a
$\mathbf{\Pi}^0_2$ subspace of $( \N \cup\{ *\} )^{ \N }$. Therefore both $\Emb(L,K)$ and $\Epi(L,K)$ are Polish
spaces.

\subsection{Some effective facts}
We assume some familiarity with basic recursion theory. For effective
descriptive set theory, we refer the reader to Section 3E of \cite{moschovakis}
or to \cite{louveau}. Notice that $\lin$ is a Polish recursive space in the
sense of Section 2.4.3 in \cite{louveau}.

We make a heavy use of Chapter
4 of \cite{moschovakis}. Let us recall the following well-known facts
(\cite{moschovakis}, 4D.3 and \cite{louveau}, Section 5.1.5, respectively). We state them in relativized form, fixing a parameter $K$.

\begin{fact}\label{Spector}
If $X$ and $Y$ are recursive spaces and $\mathsf{A} \subseteq X \times Y$ is
$\Pi^1_1(K)$ then $\{x \in X \mid \exists y \in \Delta^1_1(x,K) (x,y) \in
\mathsf{A}\}$ is $\Pi^1_1(K)$.
\end{fact}

\begin{fact}\label{uniformization}
Let $X$ be a recursive space, $Y$ a $\Pi^1_1(K)$ subset of a recursive space,
and $\mathsf{A} \subseteq X \times Y$ a $\Pi^1_1(K)$ set. Then there exists a $\Delta^1_1(K)$ function $f:\exists^\Delta \mathsf{A}\to Y$ which uniformizes $ \mathsf{A} $ on $\exists^\Delta \mathsf{A}$, where
\[
x\in\exists^\Delta \mathsf{A} \Longleftrightarrow \exists y\in\Delta^1_1(x,K)\big((x,y)\in \mathsf{A}\big).
\]
\end{fact}

Some basic operations on linear orders
are effective.

\begin{fact}
The operation $\star:L\mapsto L^\star$ is recursive.
\end{fact}

We now spell out what we mean by saying that the definition by pieces of
Lemma \ref{definitionbypieces} is $\Delta^1_1$.
\begin{fact}\label{fact:effective}
The following sets are $\Delta^1_1$:
\begin{arabiclist}
\item for $I$ an interval of $\Z$, the set of $(K,(K_i)_{i\in I})$ such
    that $K \in \lin$ and $(K_i)_i$ is a nice family
    of $K$,
\item the same with nice covering,
\item the set of triples $(K,L,M)$ such that $L$ and $M$ are adjacent
    convex subsets of $K$,
\item the set of triples $(K,L,M)$ such that $L$ and $M$ are connected
    convex subsets of $K$,
\item the set of pairs $(L,n)$ such that $n$ is the maximum of $L$.
\end{arabiclist}
\end{fact}

\begin{notation}~
\begin{itemize}
\item For $I$ an interval of $\Z$, $\Pieces_I$ stands for the set of all
    $(\bar{K}, \bar{L}) := ((K,(K_i)_{i\in I}), (L,(L_i)_{i\in I}))$ in
    $(\lin \times \lin^I)^2$ such that: $(K_i)_{i\in I}$ is a nice family
    of convex subsets of $K$; $(L_i)_{i\in I}$ is a nice covering of $L$;
    for any $i\in I$ when $K_i$ and $K_{i+1}$ are not adjacent then $L_i$
    has a maximum or $L_{i+1}$ has a minimum; and if $K_i$ and $K_{i+1}$
    are connected, so are $L_i$ and $L_{i+1}$.
\item Call $\Epi$ the space $( \N \cup\{*\})^{ \N }\cup\{\bot\}$ where
    $\bot$ is not an element of $( \N \cup\{*\})^{ \N }$; $\Epi$ is equipped with the smallest Polish topology
    extending that of $( \N \cup\{*\})^{ \N }$ and making $\{\bot\}$ a
    clopen set.
\end{itemize}
\end{notation}

\begin{fact}\label{fact:defpiecesiseffective}
For any interval $I$ of $\Z$, $\Pieces_I$ is $\Delta^1_1$ and so is the map
\begin{align*}
\DefPieces_I:\Pieces_I\times\Epi^I&\lgrao\Epi\\
\big((\bar{K},\bar{L}),(\sigma_i)_{i\in I}\big)&\lgto\begin{cases}
\sigma&\mbox{if }\forall i\in I\ \sigma_i\in \Epi(L_i,K_i) \\
\bot&\mbox{otherwise,}
\end{cases}
\end{align*}
where $\sigma\in \Epi(L,K)$ is given by Lemma \ref{definitionbypieces}.
\end{fact}

The explicit dependence on $I$ will be omitted, and we shall write simply $\DefPieces$ to denote this function.

\begin{fact}\label{fact:intervalsofscat}
If $K \in \scat$ then any convex suborder of $K$ is $\Delta^1_1(K)$.
\end{fact}
\begin{proof}
Fix $K \in \scat$. First notice that $K$ has countably many convex subsets
(see \cite{Rosens1982}, Exercise 5.33.1)). Moreover the set of convex
suborders of $K$ is a $\Pi^0_1(K)$, and hence $\Sigma^1_1(K)$, subset of
$2^{\N \times \N}$. So an application of Harrison's Effective Perfect Set
Theorem (\cite{moschovakis}, Theorem 4F.1) concludes the proof.
\end{proof}

\begin{definition}Given $\mathsf{A}$ and $\mathsf{B}$
subsets of $\lin$, denote by $\mathsf{A}+\mathsf{B}$ the
set
\[
\{K\in\lin \mid \exists K_0 (K_0 \text{ is an initial segment of } K \land
K_0 \in \mathsf{A} \land K \setminus K_0 \in \mathsf{B})\}.
\]

Define then $n\cdot \mathsf{A}$ for $n\geq1$ by induction on $n$ by $1\cdot
\mathsf{A} = \mathsf{A}$ and $(n+1)\cdot \mathsf{A}=n \cdot \mathsf{A} +
\mathsf{A}$, finally $\FinSum(\mathsf{A})$ stands for $\bigcup_{n\in\N} (n+1)
\cdot \mathsf{A}$.
\end{definition}

\begin{fact}\label{fact:coanalyticsums}
If $\mathsf{A}$ and $\mathsf{B}$ are two $\Pi^1_1(K)$
subsets of $\scat$, then so are $\mathsf{A}+\mathsf{B}$ and
$\FinSum(\mathsf{A})$.
\end{fact}

\begin{proof}
The class $\Pi^1_1(K)$ is closed by effective countable
unions, so it suffices to prove the statement for $\mathsf{A}+\mathsf{B}$.
This comes from
Facts \ref{fact:intervalsofscat} and \ref{Spector}.
\end{proof}

\subsection{Uniformizations for epimorphisms}

\begin{definition}
Given $K \in \lin$ we say that $K$
\emph{admits a $\Delta^1_1$-u\-ni\-form\-i\-za\-tion} if
there exists a $\Delta^1_1(K)$ map $\Phi_K: \lin \to \Epi$ such that
$\Phi_K(L) \in \Epi(L,K)$ when $L \strong K$ and $\Phi_K(L) = \bot$ when $L
\not\strong K$.

Given $\mathsf{A} \subseteq \lin$, if all $K \in \mathsf{A}$ admit a
$\Delta^1_1$-u\-ni\-form\-i\-za\-tion we say that $\mathsf{A}$ has the
\emph{$\Delta^1_1$-u\-ni\-form\-i\-za\-tion for epimorphisms}.
\end{definition}

Our goal is to show in Theorem
\ref{theorem:boreluniformization} below that $\scat$ has the
$\Delta^1_1$-u\-ni\-form\-i\-za\-tion for epimorphisms.

\begin{fact}\label{fact:unifstarclosed}
If $L$ admits a $\Delta^1_1$-u\-ni\-form\-i\-za\-tion, so does $L^\star$.
\end{fact}
\begin{proof}
Define $\Phi_{L^{\star }}(M)=\Phi_L(M^{\star })$.
\end{proof}

Notice that a $\Delta^1_1$-u\-ni\-form\-i\-za\-tion of
some $K \in \lin$ is a $\mathbf{\Delta}^1_1$ subset of
$\mathcal{U}=\lin\times\Epi$. Following Section 5.1.1 in \cite{louveau}, we
call $(\mathbf{D}^1, \mathbf{W}^1)$ the coding of
$\mathbf{\Delta}^1_1$-subsets of $\mathcal{U}$. Recall that we have
\begin{itemize}
\item $\mathbf{D}^1$ is a $\Pi^1_1$ subset of $ \Bai $,
    $\mathbf{W}^1$ is a $\Delta^1_1$ subset of $\mathbf{D}^1\times\mathcal{U}$
\item $\{\mathbf{W}^1_\alpha\mid \alpha\in\mathbf{D}^1\}$ is the set of
    $\mathbf{\Delta}^1_1$ subsets of $\mathcal{U}$ and for any $K$,
    $\{\mathbf{W}^1_\alpha\mid \alpha\in\mathbf{D}^1\mbox{ and } \alpha
    \mbox{ recursive in }K\}$ is the set of $\Delta^1_1(K)$ subsets of
    $\mathcal{U}$.
\end{itemize}

\begin{fact}\label{fact:uniformuniformization}
There is a partial $\Delta^1_1$ map: $\lin \to \mathbf{D}^1$, $K \mapsto
\alpha_K$ such that for  every $K \in \lin$ admitting a
$\Delta^1_1$-uniformization, $\alpha_K \in \mathbf{D}^1$ is a code of a
$\Delta^1_1$-uniformization $\Phi_K=\mathbf{W}^1_{\alpha_K}$.
\end{fact}
\begin{proof}
The relation
\[
\alpha\mbox{ is the code of a function that uniformizes } K
\]
is a $\Pi^1_1$ subset of $\mathbf{D}^1\times\lin$ by definition:
\begin{multline*}
\alpha\mbox{ codes a uniformization of } K
\Lglrao\alpha\mbox{ codes a function }\Phi\mbox{ and}\\
\forall L\in\lin\,\big( (L\not\strong K \land \Phi(L) = \bot) \lor\Phi(L)\in\Epi(L,K)\big).
\end{multline*}
The result then follows using Fact \ref{uniformization}.
\end{proof}

\begin{proposition}\label{prop:unifclosedfinitesums}
If $\mathsf{A}, \mathsf{B} \subseteq
\scat$ have the $\Delta^1_1$-u\-ni\-form\-i\-za\-tion
for epimorphisms, then so does $\mathsf{A}+\mathsf{B}$.

In particular, $\FinSum(\mathsf{A})$ has the
$\Delta^1_1$-u\-ni\-form\-i\-za\-tion for epimorphisms.
\end{proposition}
\begin{proof} We fix $K \in
\mathsf{A}+\mathsf{B}$ and $K_0 \in \mathsf{A}$, $K_1 \in \mathsf{B}$ such
that $K_0$ is an initial segment of $K$ and $K_1= K \setminus K_0$. Since
$\mathsf{A}$ and $\mathsf{B}$ have the $\Delta^1_1$-u\-ni\-form\-i\-za\-tion
for epimorphisms, for $i = 0,1$ there exists a
$\Delta^1_1$-u\-ni\-form\-i\-za\-tion $\Phi_{K_i}$ for $K_i$. As $K_i$ is a
convex subset of $K \in \scat$, by Fact \ref{fact:intervalsofscat}
$\Delta^1_1(K_i)$ is included in $\Delta^1_1(K)$. Recalling that each
$\Phi_{K_i}$ is a $\Delta^1_1(K_i)$ map, this implies that $\Phi_{K_i}$ is a
$\Delta^1_1(K)$ map.

By Lemma \ref{definitionbypieces} (and using the fact that $K_0$ and $K_1$
are adjacent but not connected), $L$ satisfies $L\strong K$ if and only if
there is a nice covering $(L_0,L_1)$ of $L$ satisfying $L_0 \strong K_0$ and
$L_1 \strong K_1$. Since $\Phi_{K_i}$ allows to check whether $L_i \strong
K_i$ in a $\Delta^1_1(K)$ way and using Fact \ref{fact:effective}(2), the set
$\mathsf{C}$ of triples $(L,L_0,L_1)$ satisfying the latter is
$\Delta^1_1(K)$.

We now use the effective version of Lusin-Novikov's \lq\lq small
section\rq\rq\ uniformization result (see \cite{moschovakis}, 4F.6: the
statement there is not effective, but the hint proves the effective version)
to obtain two $\Delta^1_1(K)$ functions $\Psi_0$ and $\Psi_1$ with domain
$\{L \mid \exists L_0,L_1\, (L,L_0,L_1) \in \mathsf{C}\}$ such that $(L,
\Psi_0(L), \Psi_1(L))\in \mathsf{C}$ holds for any $L \strong K$.

We can now define $\Phi_K: \lin \to \Epi$ by setting $\Phi_K(L)$ to be the
map defined by pieces from $\Phi_{K_0} \circ \Psi_0(L)$ and $\Phi_{K_1} \circ
\Psi_1(L)$ when $\Psi_0(L)$ and $\Psi_1(L)$ are defined, and $\Phi_K(L) =
\bot$ otherwise. By Fact \ref {fact:defpiecesiseffective}, $\Phi_K$ is a
$\Delta^1_1(K)$ map and in fact a $\Delta^1_1$-u\-ni\-form\-i\-za\-tion of
$K$.
\end{proof}

We recall (some version of) Hausdorff's hierarchy of
countable scattered linear orders.
\begin{itemize}
\item Call $\scat_0\subseteq \lin$ the class of singleton
    orders,
\item an element of $ \lin $ is in $\scat_\alpha$ when it is isomorphic to
    a  finite sum, an $\omega$-sum or an $\omega^{\star }$-sum of elements
    of $\bigcup_{\beta<\alpha}\scat_\beta$.
\end{itemize}
Hausdorff proved that $\scat=\bigcup_{\alpha<\omega_1}\scat_\alpha$, so for
$K\in\scat$ we define the \emph{Hausdorff rank} of $K$:
\[
\Hrk(K)=\min\{\alpha<\omega_1\mid K\in\scat_\alpha\}.
\]
We have $\scat_\alpha = \{K \in \scat \mid \Hrk(K) \leq
\alpha\}$ and we can set $\scat_{<\alpha}= \bigcup_{\beta<\alpha} \scat_\beta
= \{K \in \scat \mid \Hrk(K) < \alpha\}$.

Recall that if $L$ is a suborder of $K$, then $\Hrk(L)\leq\Hrk(K)$ holds, and
that $\Hrk$ is a $\Pi^1_1$-norm (see \cite{moschovakis}, Section 4B). In
particular, $\scat_\alpha$ and $\scat_{<\alpha}$ are
$\Delta^1_1(\alpha)$. Moreover, for any $K \in \scat$, using Theorem
4D.1(iii) of \cite{moschovakis} we have $\Hrk(K) \in \Delta^1_1(K)$.\medskip

To prove that if $\scat_\alpha$ has the
$\Delta^1_1$-u\-ni\-form\-i\-za\-tion for epimorphism so does
$\scat_{\alpha+1}$ we only need to handle the case of $\omega$-sums (as for
the case of $\omega^{\star}$-sums we can use Fact
\ref{fact:unifstarclosed}, and the case of finite sums is handled by
Proposition \ref{prop:unifclosedfinitesums}). We use the following
notion, implicitly used in \cite{Landra1979}.

\begin{definition}
We say that an order $K$ is \emph{stable} if and only if for all $k\in K$ we
have $K\equiv_s[k,{\rao})_K$.
\end{definition}

Since it is always the case that $[k,{\rao})_K \strong K$, $K$ stable really
means $K \strong [k,{\rao})_K$ for all $k \in K$.
Notice that in particular a stable $K$ has a minimum and that the only stable
orders with a maximum are the singletons.

Recall from \cite{Landra1979} and \cite{cacama} that the class of countable linear
orders is well-quasi-ordered (wqo) under epimorphisms. In particular, $(\lin,
{\strong})$ is well-founded. We use the following observation, due to
Landraitis \cite{Landra1979}, Lemma 2.2.

\begin{lemma}\label{lem:landraitistrick}
Every $K\in\lin$ has a stable final segment.
\end{lemma}

\begin{proof}
Fix $K\in\lin$ and look for $a\in K$ such that $[a,{\rao})_K$ is
stable.
Fix a sequence $(a_n)_{n\in\N}$ monotone and cofinal in $K$. The sequence
$([a_n,{\rao})_K)_{n\in\N}$ is $\strong $-decreasing so there is $N\in\N$
such that $([a_n,{\rao})_K)_{n\geq N}$ is $\equiv_s$-constant, for $\strong $
is well-founded on $\lin$. Choose then $a=a_N$. For
any $k\geq_Ka$ there is $n>N$ such that $k\leq_Ka_n$ holds, and finally
\[
{[a,{\rao})_K} \strong {[a_n,{\rao})_K} \strong [k,{\rao})_K,
\]
so $[a,{\rao})_K$ is a stable final segment of $K$.
\end{proof}

We need the following characterization of
stability (also essentially contained in \cite{Landra1979}, Lemma 2.2).

\begin{fact}\label{stablesimplerexpression}
An admissible linear order $K$ is stable if and only if it has a minimum and
for all $a_0\leq_K a_1$ and $a$ in $K$ there are $b_0,b_1$ in $K$ such that
$a\leq_Kb_0\leq_Kb_1$ and $[a_0,a_1]_K \strong [b_0,b_1]_K$.
\end{fact}

\begin{proof}
It is immediate that if $K$ is stable then it has the desired property, so it
is enough to show the converse. Given any $k\in K$, take $(a_i)_{i\in \N }$ a
cofinal monotone sequence in $K$, with $a_0=\min K$ and $a_1=k$. Use the
hypothesis to find a sequence $(b^0_i, b^1_i)_{i\in \N }$ such that we have
\begin{itemize}
\item $a_{i+1}\leq_Kb^0_i\leq_Kb^1_i\leq_Kb^0_{i+1}$,
\item $[a_i,a_{i+1}]_K\strong [b^0_i,b^1_i]_K$.
\end{itemize}
Notice that since $[b^0_0,
b^1_0]_K \strong [k, b^1_0]_K$ we can assume that $b^0_0=k$.

We can apply Lemma \ref{definitionbypieces} to define by pieces a surjection
showing that $K \strong [k,{\rao})_K$.
\end{proof}

\begin{theorem}\label{theorem:boreluniformization}
$\scat$ has the $\Delta^1_1$-u\-ni\-form\-i\-za\-tion for epimorphisms.
\end{theorem}

\begin{proof}
We prove inductively on $\alpha$ that $\scat_\alpha$ has the
$\Delta^1_1$-u\-ni\-form\-i\-za\-tion for epimorphisms. Take $K\in\scat_0$,
so that there is $n\in\N$ such that $K=\{n\}$. For $m\in \N $ define
$\tau_{n,m}: \N \rao \N \cup\{ *\} $ by letting $\tau_{n,m}(n)=m$, and
$\tau_{n,m}(i)=*$ if $i\neq n$; then
\begin{align*}
\Phi_K:\lin&\lgrao \Epi \\
L&\lgto\begin{cases}
\tau_{n,m}&\mbox{if }L=\{m\} \mbox{ for some } m \\
\bot&\mbox{otherwise}
\end{cases}
\end{align*}
is a $\Delta^1_1$-u\-ni\-form\-i\-za\-tion\footnote{Notice that this $\Phi_K$
is not continuous!} of $K$.\medskip

Fix now $\alpha<\omega_1$ with $\alpha \geq 1$, and suppose that
$\scat_{<\alpha}= \bigcup_{\beta<\alpha} \scat_\beta$ has the
$\Delta^1_1$-u\-ni\-form\-i\-za\-tion for epimorphisms. We need to prove that
$\scat_\alpha$ has it too.

By Proposition
\ref{prop:unifclosedfinitesums}, $\FinSum(\scat_{<\alpha})$ has the
$\Delta^1_1$-u\-ni\-form\-i\-za\-tion for epimorphisms. Recall also that, by Fact
\ref{fact:uniformuniformization}, there is a
$\Delta^1_1$ map that for any $L \in
\FinSum(\scat_{<\alpha})$ chooses the code $\alpha_L$ of a
$\Delta^1_1$-u\-ni\-form\-i\-za\-tion $\Phi_L$ for $L$.

Calling $\mathsf{A}$
the set of stable elements of $\scat_{\alpha}$,
Lemma
\ref{lem:landraitistrick} yields that $\scat_\alpha$ can be defined as
\[
\FinSum(\scat_{<\alpha})\cup\big(\FinSum(\scat_{<\alpha})+\mathsf{A}\big)\cup\star\big(\FinSum(\scat_{<\alpha})+\mathsf{A}\big),
\]
where we are using obvious notations.

Fact \ref{fact:unifstarclosed} and Proposition
\ref{prop:unifclosedfinitesums} tell us that if $\mathsf{A}$ has the
$\Delta^1_1$-u\-ni\-form\-i\-za\-tion for epimorphisms, then so does
$\scat_{\alpha}$.

Fix $K \in \mathsf{A}$ and let $k_0$ be the minimum of $K$. If $K$ has a
maximum, it is a singleton and we already know that it
has a $\Delta^1_1$-u\-ni\-form\-i\-za\-tion. Thus we can assume that $K$ is a
stable $\omega$-sum of elements of $\scat_{<\alpha}$. To define a
$\Delta^1_1(K)$-uniformization $\Phi_K$ of $K$ we use the
$\Delta^1_1$-u\-ni\-form\-i\-za\-tions $\Phi_{[k,k']_K}$ of $[k,k']_K$.
Notice that the $\Phi_{[k,k']_K}$ are $\Delta^1_1(K)$ as well, because
intervals are $\Delta^0_1 (K)$.

We now define $\Phi_K (L)$ distinguishing three cases, each of them defined
by a $\Delta^1_1$ property.

\smallskip

If $L$ has no minimum (a $\Pi^0_2$ condition) then notice that $L \not\strong
K$ and let $\Phi_K(L)=\bot$.

If $L$ has two extrema (a $\Sigma^0_2$ condition) then notice that $L \strong
K$ if and only if there is $k_1 \in K$ such that $L \strong [k_0,k_1]_K$
holds, if and only if $\exists k_1 \in K\, \Phi_{[k_0,k_1]_K} (L) \neq \bot$.
In case $k_1$ does exist, choose $k_1$ minimal (as a natural number) and
define $\Phi_K(L)$ by pieces from $\Phi_{[k_0,k_1]_K} (L)$ using the function
$\DefPieces$ of Fact \ref{fact:defpiecesiseffective}. Otherwise let
$\Phi_K(L)=\bot$.

The last case is when $L$ has minimum but no maximum (a $\Delta^0_3$
condition). Let $\mathsf{B} = \{L \in \lin \mid \text{$L$ has minimum but no
maximum}\}$. If $L \in \mathsf{B}$ we denote by $\{l_i \mid i \in \N\}$ the
canonical cofinite sequence in $L$, defined by letting $l_0$ be the minimum
of $L$, and $l_{i+1}$ be the least (as natural number) $l \in L$ such that
$l_i <_L l$. Notice that the map $L \mapsto \{l_i \mid i \in \N\}$ is
$\Delta^1_1$ on the $\Delta^0_3$ set $\mathsf{B}$.

\begin{claim}
There exists a $\Delta^1_1(K)$ function $f: \mathsf{B} \to K^\N$ such that,
writing $f(L)(i) = k_i^L$, we have that the sequence $\{k_i^L \mid i\in\N\}$
is strictly increasing and cofinal in $K$ and moreover $L \strong K$ if and
only if $[l_i,l_{i+1}]_L \strong [k_i^L, k_{i+1}^L]_K$ for every $i \in \N$.
\end{claim}
\begin{proof}
Given $L \in \mathsf{B}$ we uniformly define in a $\Delta^1_1(K)$ way the
sequence $\{k_i^L \mid i\in\N\}$ and an auxiliary sequence $\{m_i^L \mid
i\in\N\} \in 2^{ \N }$ by induction on $i$. The intuition for $m_i^L$ is that
as long as $m_i^L = 0$ we are still hoping to show that $L \strong K$, while
when we set $m_i^L = 1$ we actually know that $L \not\strong K$ and we just
need to make sure that $\{k_i^L \mid i\in\N\}$ is cofinal in $K$.

As $K \in \mathsf{B}$, $\{k_i\mid i\in\N\}$ stands for the canonical cofinite
sequence in $K$. First, $k_0^L=k_0$ is the minimum of $K$ and $m_0^L=0$.
Assuming we have already defined $k_i^L$ and $m_i^L$ we proceed as follows.
If $m_i^L = 0$ we look for $k \in K$ such that $[l_i,l_{i+1}]_L \strong
[k_0^L, k]_K$. If we succeed, we let $k$ be the least (as natural number)
such $k$ and, using the stability of $K$ and Fact
\ref{stablesimplerexpression}, find the least (code for) a pair
$(\hat{k},k_{i+1}^L) \in K^2$ such that $[k_0^L,k]_K \strong
[\hat{k},k_{i+1}^L]_K$ and $\max_K (k_i^L, k_i) \leq_K \hat{k}$. This way we
defined $k_{i+1}^L$, and we set also $m_{i+1}^L = 0$. Notice that in this
case we have
\[
[l_i,l_{i+1}]_L \strong [k_0^L,k]_K \strong [\hat{k},k_{i+1}^L]_K \strong [k_i^L,k_{i+1}^L]_K.
\]
If either the search for $k \in K$ such that $[l_i,l_{i+1}]_L \strong [k_0^L,
k]_K$ fails or $m_i^L = 1$, we let $m_{i+1}^L = 1$ and $k_{i+1}^L$ be the
least (as natural number) $k \in K$ such that $\max_K (k_i^L, k_i) <_Kk$.

Since we made sure that $k_i \leq_K k_i^L$ the sequence $\{k_i^L \mid
i\in\N\}$ is indeed cofinal in $K$.

Now notice that if $m_i^L = 0$ for every $i$ then $[l_i,l_{i+1}]_L \strong
[k_i^L,k_{i+1}^L]_K$ for every $i$; using a definition by pieces we find a
witness to $L \strong K$. If instead $m_i^L = 1$ for some $i$ let $i$ be the
least such. Then $[l_i,l_{i+1}]_L \not\strong [k_i^L,k_{i+1}^L]_K$ and there
is no $k \in K$ such that  $[l_i,l_{i+1}]_L \strong [k_0^L, k]_K$. The latter
fact implies $L \not\strong K$.
\end{proof}

Now, using the claim, we can define $\Phi_K$ on $\mathsf{B}$. If
$[l_i,l_{i+1}]_L \strong [k_i^L,k_{i+1}^L]_K$ for every $i$ (a
$\Delta^1_1(K)$ condition) then $\Phi_K(L)$ can be defined applying the
function $\DefPieces$ from Fact \ref{fact:defpiecesiseffective} to the
epimorphisms $\Phi_{[k_i^L,k_{i+1}^L]_K} ([l_i,l_{i+1}]_L)$. If instead
$[l_i,l_{i+1}]_L \not\strong [k_i^L,k_{i+1}^L]_K$ for some $i$ we set
$\Phi_K(L)= \bot$. \hfill $\square$
\end{proof}

\medskip

We can finally pinpoint the complexity of $\StS$, but first the complexity of
$\StS\cap\scat$.

\begin{corollary}\label{cor:scatstrongsurj}
The set of scattered strongly surjective orders is $\Pi^1_1$ and
$\mathbf{\Pi}^1_1$-complete.
\end{corollary}
\begin{proof}
The fact that $\StS \cap \scat$ is
$\mathbf{\Pi}^1_1$-hard is contained in Proposition \ref{basichardness1}.

Given an order $K$ say that a $\Delta^1_1$-u\-ni\-form\-i\-za\-tion $\Phi$ of
$K$ is strong if for all $L$ such that $L \inj K$ we have $\Phi(L) \neq
\bot$. By Theorem \ref{theorem:boreluniformization} $K$ is scattered and
strongly surjective if and only if it is scattered and admits a strong
$\Delta^1_1$-u\-ni\-form\-i\-za\-tion.

This gives in turn, using Fact \ref{Spector}, a
$\Pi^1_1$ definition of $\StS \cap \scat$.
\end{proof}

\begin{corollary}\label{cor:upperbound}
The set $\StS$ is the union of a $\Sigma^1_1$ set and a $\Pi^1_1$ set. It is
in particular $ \check{\mathrm{D}}_2(\mathbf{\Pi}^1_1)$, and in fact $
\check{\mathrm{D}}_2(\mathbf{\Pi}^1_1)$-complete.
\end{corollary}
\begin{proof}
By Proposition \ref{basichardness2}, Corollary \ref{cor:scatstrongsurj}
and Theorem \ref{main}.
\end{proof}


\section{Looking for uncountable strongly surjective orders} \label{unctbless}

\subsection{Classical examples are not strongly surjective}\label{classics}
Recall that Proposition \ref{basics2}(4) states that a strongly surjective
linear order can have at most the cardinality of the continuum. Here we show
that the most common orders of size the continuum and those that can be
obtained from them using basic operations are not strongly surjective. We use
different techniques, and for some linear orders we have different proofs
that they are not strongly surjective.

We first give a cardinality obstruction for strong surjectivity of suborders
of $\R$.

\begin{theorem}\label{prop:CHnegative}
Let $\aleph_0< \kappa \leq 2^{\aleph_0}$, and assume $2^{\aleph_0}<
2^{\kappa}$. Then no $X \subseteq \R$ of cardinality $\kappa$ can be strongly
surjective.
\end{theorem}
\begin{proof}
We use a counting argument: there are more subsets of $X$ than
order-preserving  maps from $X$ to $X$.

Since $X\subseteq \R$ we can find a countable subset $D$ of $X$ such that for
all $x, x' \in X$, if $x<x'$ holds then there is a $d\in D$ with $x\leq d\leq
x'$, and such that moreover the endpoints of $X$ belong to $D$, if they
exist. Every order-preserving map from $X$ to $X$ is the extension of an
order-preserving map from $D$ to $X$, and there are at most continuum many of
those:
\[
\big|\{f:D\to X\mid f\mbox{ order-preserving} \} \big | \leq \left|X^D\right| =\kappa^{\aleph_0}=2^{\aleph_0}.\]
Fix now $f:D\to X$ order-preserving, and compute how many order-preserving
extensions $g:X\to X$ of $f$ there can be. Take $x\in X\setminus D$, we can
pick for $g(x)$ any $y\in X$ that satisfies
\begin{equation}\label{CHequ:1}
\forall d\in ({\lao},x)_X\cap D\ \forall d'\in (x,{\rao})_X\cap D\ f(d)\leq y\leq f(d').
\end{equation}
Call $I_x$ the convex set of all points $y$ satisfying (\ref{CHequ:1}). If
$I_x$ is trivial, that is empty or reduced to a singleton, then there is at
most one order-preserving extension of $f$ to $x$. Notice now that by the
properties of $D$, for $x\neq x'$ in $X\setminus D$ the sets $I_x$ and
$I_{x'}$ are disjoint. Since $X\subseteq \R$ there can be only countably many
non-trivial $I_x$, and each of these yields at most $\kappa$ possible
extensions, so for a fixed $f$ we have
\[\big|\{g:X\to X\mid g\mbox{ order-preserving and }f\subset g\}\big|\leq\kappa^{\aleph_0}=2^{\aleph_0}.\]
All in all we have at most continuum many order-preserving maps from $X$ to
$X$, but $X$ has even more subsets by hypothesis so it cannot be strongly
surjective.
\end{proof}

\begin{corollary}
$\R$ and $\R \setminus \Q$ are not strongly surjective.
\end{corollary}

Theorem \ref{prop:CHnegative} and its proof do not provide a concrete $L$
such that $L \inj \R$ yet $L \not\strong \R$ (and similarly for $\R \setminus
\Q$). A useful technique to prove that a linear order does not admit
epimorphisms onto another one is to compare their gaps. Recall that a
\emph{gap} of $K$ is given by a non-empty initial segment $A \sub K$ with no
maximum such that $K \setminus A$ is non-empty and has no minimum. Let $G(K)$
be the set of gaps of $K$ linearly ordered by $\subseteq$.

\begin{proposition}
If $L$ and $K$ are linear orders such that $L \strong K$ then $G(L) \inj
G(K)$.
\end{proposition}
\begin{proof}
If $f$ is an epimorphism from $K$ onto $L$ then $A \mapsto f^{-1}(A)$ is an
injection from $G(L)$ to $G(K)$.
\end{proof}

\begin{corollary}\label{cor:gaps}
Any linear order $L$ with $|G(L)|< 2^{\aleph_0}$ is such that $\Q \not\strong
L$.

Hence $\Q$ witnesses the fact that $\R$, $\R \setminus \Q$, $\Z^\N$, and, for
every countable $\alpha \geq \omega$, $2^\alpha$ ordered lexicographically
are not strongly surjective.
\end{corollary}
\begin{proof}
The first part follows immediately from the Proposition because $|G(\Q)| =
2^{\aleph_0}$.

Each of the linear orders considered in the second part of the statement (and
indeed each short uncountable linear order) is non-scattered. Now observe
that $ \R $ and $2^\alpha $ are complete (that is, they have no gaps), while $|G(\R \setminus \Q)| = |G(\Z^\N)| =
\aleph_0$. The statements about $\R$ and $\R \setminus \Q$ are obvious.

To see that $2^\alpha$ is complete let $A$ be a non-empty subset of $2^{\alpha}$. Define inductively $x\in 2^{\alpha }$
as follows: given $\beta\in\alpha $ and assuming that $x(\gamma )$ has been
constructed for every $\gamma\in\beta $, let $x(\beta )\in\{ 0,1\} $ be the
least value such that $x|_{\beta +1}$ majorizes $\{ z|_{\beta +1}\mid z\in
A\} $. Then $x$ majorizes $A$. Moreover, if $y$ majorizes $A$ and
$y<_{lex}x$, let $s=y\cap x$; then $s0\subseteq y,s1\subseteq x$, contrary to
the definition of $x$. So $x=\sup A$.

To see that $\Z^\N$ has only countably many gaps let $L= \Z^\N \cup (
\Z^{< \N }\setminus\{\es\} )$ where the order on $\Z^\N$ is extended to
$L$ by ordering $ \Z^{< \N }\setminus\{\es\} $ lexicographically and
letting $x <_L s$ if and only if $x|_{\lh (s)} \leq_{lex} s$ for every $x \in
\Z^\N$ and $s \in \Z^{< \N }\setminus\{\es\} $. Then $L$ is complete.
\end{proof}

Notice that the fact that $2^\N$ and $\Z^\N$ are not strongly surjective
shows that Corollary \ref{multiple}, stating that strongly surjective orders
are closed under finite products, cannot be extended to infinite products.

The next natural candidates for being uncountable strongly surjective orders
are the finite products obtained by using $\R \setminus \Q$, $\R$ and
possibly some countable orders as factors. We show however that no
uncountable strongly surjective order can be obtained in this way.

\begin{lemma}\label{fact:negativeproduct}
Let $K$, $L$ and $M$ be linear orders. Suppose that $K \not\strong L$ and
that $K'\not\strong M$ for any convex subset $K'$ of $K$ that has more than
one point. Then we have $K \not\strong ML$.
\end{lemma}
\begin{proof}
Suppose we have $\varphi \in \Epi(K,ML)$ and consider $f: L \rao
\mathcal{P}(K)$ defined by $f(r) = \varphi (M\times\{r\})$. There is some
$r\in L$ such that $f(r)$ is not a singleton, otherwise $f$ would induce an
epimorphism from $L$ onto $K$. But then $f(r) \strong M$ and $f(r)$ is a
convex subset of $K$ with more than one point, which is again impossible.
\end{proof}

\begin{definition}
If $\kappa $ is an infinite cardinal, a linear order $L$ is $\kappa $-{\it
dense} if it has no end points and between any two distinct elements of $L$
there are exactly $\kappa$ elements of
$L$.
\end{definition}

\begin{lemma} \label{manypointsmanygaps}
There is a $2^{\aleph_0}$-dense suborder $M$ of $\R \setminus \Q$ such that
every interval $[x,y]_M$, for $x<_M y$, has $2^{\aleph_0}$ gaps.
\end{lemma}
\begin{proof}
The following construction is a variation on the classical construction of a
Bernstein set (see e.g.\ Example 8.24 in \cite{cdst}).

Let $\{ I_n\}_{n\in \N }$ be the set of the traces on $\R \setminus \Q$ of
the real intervals with rational endpoints. In each $I_n$ we define subsets
$X_n=\{ x_{\beta }^n\}_{\beta\in 2^{\aleph_0}}$ and $Y_{\beta }^n=\{ y_{\beta
m}^n\}_{m\in \N }$, for $\beta\in 2^{\aleph_0}$, such that the elements
$x_{\beta }^n,y_{\beta m}^n$ are all distinct. Fix $\beta <2^{\aleph_0}$ and
$n\in \N $. Suppose that $x_{\beta '}^{n'}$ and $y_{\beta 'm}^{n'}$ are
defined for all $\beta '<\beta $ and $n',m\in \N $, as well as $x_{\beta
}^{n'}$ and $y_{\beta m}^{n'}$ for $n'<n$ and $m\in \N $. Notice that the set
$A$ of these elements, if non-empty, has cardinality $\max (\aleph_0,|\beta
|)< 2^{\aleph_0}$. So pick any $x_{\beta}^n\in I_n\setminus A$ and distinct
$y_{\beta m}^n\in I_n\setminus (A\cup\{ x_{\beta }^n\} )$ such that
\[
\sup\{ y_{\beta m}^n\mid m\in \N ,y_{\beta m}^n<x_{\beta }^n\} =x_{\beta }^n=
\inf\{ y_{\beta m}^n\mid m\in \N ,x_{\beta }^n<y_{\beta m}^n\} .
\]
It follows that every $x_{\beta }^n$ is a gap in $M=( \R \setminus \Q
)\setminus\{ x_{\beta }^n\mid\beta <2^{\aleph_0},n\in \N \} $. Moreover, if
$x,y\in M,x<y$, then $[x,y]_M$ contains some $I_n\cap M$, which has the
cardinality of the continuum (containing all $y_{\beta m}^n$) and has
continuum many gaps (at least every $x_{\beta }^n$).
\end{proof}

\begin{theorem}\label{unctblclassic}
Let $L = \prod_{0\leq i\leq n} L_i$, where for each $i$ either $L_i$ is
countable or $\R \setminus \Q \inj L_i$ and $|G(L_i)| < 2^{\aleph_0}$ ($\R
\setminus \Q$ and $\R$ are instances of such linear orders). If $L$ is
uncountable, then $L$ is not strongly surjective.
\end{theorem}
\begin{proof}
Suppose that $L$ is uncountable, so at least one of the factors is
uncountable. Let $M$ be the order given by Lemma \ref{manypointsmanygaps}.
Since $M \inj \R \setminus \Q$ we have $M \inj L$. It then suffices to show
that $M \not\strong L$, and this can be done by induction on $n$.

For $n=0$, notice that $M \not\strong L$ since $|G(L)| < 2^{\aleph_0} =
|G(M)|$.

If the statement holds for $n$, let $L = \prod_{0\leq i\leq n+1} L_i$. If
$\prod_{1 \leq i \leq n+1} L_i$ is countable (by cardinality reasons), or by
inductive hypothesis, $M \not\strong \prod_{1\leq i\leq n+1} L_i$. Moreover,
if $M'$ is any convex subset of $M$ containing more than one point, then
$M'\nleq_sL_0$ either by cardinality reasons (if $L_0$ is countable) or by
the fact that $M'$ has more gaps than $L_0$ (if $L_0$ is uncountable). Now
apply Lemma \ref{fact:negativeproduct}.
\end{proof}

We already argued that $2^\N$ and $\Z^\N$ ordered lexicographically are not
strongly surjective. However the gap method does not apply to other natural
infinite lexicographic products, such as $\Q^\N$, $\R^\N$, and $(\R \setminus
\Q)^\N$, which have $2^{\aleph_0}$ many gaps. First we show that products
such as $\R^\N$ and $(\R \setminus \Q)^\N$ are not strongly surjective.

\begin{theorem}\label{lexprod}
For every $k \in \N$ let $L_k$ be a linear order with at least two elements
such that for every convex set $K \subseteq L_k$ we have $\Q \not\strong K$.
Then $\Q \not\strong \prod_{k \in \N} L_k$, where the product is ordered
lexicographically.
\end{theorem}
\begin{proof}
First of all notice that the hypothesis on $L_k$ implies that every convex
subset of $L_k$ does not surject onto $1+\Q$, $\Q+1$, or $1+\Q+1$.

Suppose $f:\prod_{k \in \N} L_k \rao \Q$ is an epimorphism. Our goal is to
define an embedding $g: 2^\N \rao \prod_{k \in \N} L_k$ such that $f$ is
injective on the range of $g$, thus reaching a contradiction because $\Q$ is
countable.

If $s \in \prod_{i<k} L_i$ for some $k \in \N$ we let $N_s = \{z \in \prod_{k
\in \N} L_k \mid s \subset z\}$. To define $g$ we define $h: 2^{<\N} \rao
\bigcup_{k \in \N} \prod_{i<k} L_i$ such that $h(s) \in \prod_{i<2k} L_i$
when $s \in 2^k$, $h(s) \subset h(t)$ when $s \subset t$, $f(N_{h(s)})$ is a
convex subset of $\Q$ with at least two elements, and $f(N_{h(s \conc 0)})
\cap f(N_{h(s \conc 1)})=\emptyset $. Then we set $g(x) = \bigcup_{k \in \N} h(x \res
k)$, so that it is immediate that if $x,y \in 2^\N$ are distinct then
$f(g(x)) \neq f(g(y))$.

The definition of $h(s)$ is by recursion on the length of $s$, starting from
$h(\es) = \es$. Thus we assume that $s$ has length $k$ and $h(s)$ is defined
respecting the conditions, so that $f(N_{h(s)})$ is isomorphic to one of
$\Q$, $1+\Q$, $\Q+1$, and $1+\Q+1$. Consider the map that sends $\ell \in
L_{2k}$ to $f(N_{h(s) \conc \ell}) \subseteq \Q$: since $f(N_{h(s)})
\not\strong L_{2k}$, for some $\ell_0 \in L_{2k}$ we have that $f(N_{h(s)
\conc \ell_0})$ is not a singleton. Since $L_{2k}$ has at least two elements
$\ell_0$ is either not the maximum or not the minimum of $L_{2k}$. Let us
assume it is not the maximum (otherwise we reason symmetrically). Then
$(\ell_0, {\rao})_{L_{2k}}$ is a nonempty convex subset of $L_{2k}$ and hence
does not surject onto $\bigcup_{\ell >_{L_{2k}} \ell_0} f(N_{h(s) \conc
\ell})$, which is isomorphic to one of $\Q$, $1+\Q$, $\Q+1$, and $1+\Q+1$.
Therefore we can find $\ell_1 >_{L_{2k}} \ell_0$ such that $f(N_{h(s) \conc
\ell_1})$ is also not a singleton. Notice that it might be that $f(N_{h(s)
\conc \ell_1})$ intersect $f(N_{h(s) \conc \ell_0})$ in a common
endpoint. Thus we go to the next level and find $\ell_{00} <_{L_{2k+1}}
\ell_{01} <_{L_{2k+1}} \ell_{10} <_{L_{2k+1}} \ell_{11}$ such that $f(N_{h(s)
\conc \ell_i \ell_{ij}})$ is not a singleton for every $i,j$. Consequently
$f(N_{h(s) \conc \ell_0 \ell_{00}}) \cap f(N_{h(s) \conc \ell_1 \ell_{11}}) =
\es$. Let $h(s \conc i) = h(s) \conc \ell_i \ell_{ii}$ for $i=0,1$.
\end{proof}

\begin{corollary}\label{cor:lexprod}
$\Q$ witnesses that $2^\N$, $\Z^\N$, $\R^\N$, and $(\R \setminus \Q)^\N$
ordered lexicographically are not strongly surjective.
\end{corollary}
\begin{proof}
First notice that $2$, $\Z$, $\R$, and $\R \setminus \Q$ satisfy the
condition imposed by Theorem \ref{lexprod} on the $L_i$'s. Then observe that
each of $2^\N$, $\Z^\N$, $\R^\N$, and $(\R \setminus \Q)^\N$ is
non-scattered.
\end{proof}

To show that $\Q^\N$ is not strongly surjective we must use a different
approach (obviously $\Q \strong \Q^\N$): we exploit the definability of
epimorphisms in certain settings.

\begin{theorem} \label{nocantor}
No uncountable Borel suborder of $2^{ \N }$, with the lexicographic order, is
strongly surjective.
\end{theorem}

\begin{proof}
First, notice that the usual product topology and the order topology on $2^{
\N }$ coincide. Indeed, for any $s\in 2^{<\N}$, the basic open set $N_s=\{
x\in 2^{ \N }\mid s\subseteq x\} $ is open in the order topology:
\begin{itemize}
\item[{\bf -}] $N_{\es }=2^{ \N }$
\item[{\bf -}] if $s=s'01^h$ for some $h>0$, then
    $N_s=(s'01^{h-1}01^{\infty },s'10^{\infty })_{2^{ \N }}$; similarly if
    $s=s'10^h$ for some $h>0$
\item[{\bf -}] if $s=0^h$ for some $h>0$, then $N_s=(\leftarrow
    ,0^{h-1}10^{\infty })_{2^{ \N }}$; similarly if $s=1^h$ for some $h>0$
\end{itemize}
Conversely, given any $x\in 2^{ \N }$, fix $y\in (x,\rightarrow )_{2^{ \N }}$
and set $s=x\cap y$. Observe that $x\in N_{s0}$ and $y\in N_{s1}$, so
$N_{s1}\subseteq (x,\rightarrow )_{2^{ \N }}$, which implies that
$(x,\rightarrow )_{2^{ \N }}$ is open in Cantor space. Similarly one proves
that $(\leftarrow ,x)_{2^{ \N }}$ is open.

Now let $X\subseteq 2^{ \N }$, and fix any order-preserving function $f:X\to
2^{ \N }$. For any $s\in 2^{<\N}$ the subset $f^{-1}(N_s)$ of $X$ is convex,
so there exists a convex subset $A$ of $2^{ \N }$ such that
$f^{-1}(N_s)=A\cap X$. By the above, if $x\in A$ and $x$ is not an end point
of $A$, then $x$ is in the topological interior of $A$. This implies that $A$
is the union of an open set plus at most two points (its end points, if they
exist), so $A$ is Borel in $2^{ \N }$ and $f^{-1}(N_s)$ is Borel in $X$;
consequently, $f$ is Borel.

If $X$ is Borel in $2^{ \N }$, then $f(X)$ must be analytic; so if $X$ is
uncountable, then $X$ is not strongly surjective, since there exist
non-analytic subsets of $X$ onto which there can be no epimorphism.
\end{proof}

\begin{corollary} \label{noznqn}
$2^\N$, $\Z^\N$ and $\Q^\N$ ordered lexicographically are not strongly
surjective.
\end{corollary}
\begin{proof}
It is enough to show that $\Z^\N$ and $\Q^\N$, endowed with the product
topology on the discrete topology, Borel embed order preservingly in $2^\N$
ordered lexicographically. Since the natural inclusion is a Borel order
preserving embedding of $\Z^\N$ into $\Q^\N$, by the main theorem of
\cite{louveauold} it is enough to prove that $2^{\omega +1} \not \inj \Q^\N$.
Notice that $2^{\omega +1}$ has uncountably many pairs of consecutive points,
so if $2^{\omega+1} \inj \Q^\N$ then $\Q^\N$, being dense, should have
uncountably many pairwise disjoint open intervals. However this is not the
case, as every $\leq_{lex}$-open interval in $\Q^\N$ contains a non-empty
open subset in the Polish topology of $\Q^\N$.
\end{proof}

The argument of Theorem \ref{nocantor} cannot be extended to suborders of
$2^{\alpha}$ for $\omega +1 \leq \alpha <\omega_1$, since on $2^{\alpha}$ the
order topology and the product topology do not coincide. In fact, for $\omega
+1 \leq \alpha <\omega_1$, there are more order-preserving functions from
$2^{\alpha}$ into itself than Borel maps with respect to the product
topology. To see this, for every $A \subseteq 2^{\N} $ consider the function
$\varphi_A: 2^{\alpha} \to 2^{\omega+1}$ defined by letting
\[
\varphi_A(x)=
\begin{cases}
(x \res \omega) \conc 0 & \text{if } x \res \omega \in A; \\
(x \res \omega) \conc 1 & \text{if } x \res \omega \notin A.
\end{cases}
\]
The function $\varphi_A$ is order-preserving, and $\varphi_A \neq \varphi_B$
whenever $A \neq B$. Since $2^{\omega +1}$ is isomorphic to a suborder of
$2^{\alpha }$, this shows that there are at least $2^{2^{\aleph_0}}$
order-preserving functions from $2^{\alpha }$ into itself.

\subsection{Beyond ZFC}\label{unctblessorders}
In contrast with the negative results that we showed so far, we now build an
uncountable strongly surjective order under extra set theoretic assumptions.
For $\kappa$ an infinite cardinal less than the continuum, consider the
statement
\begin{center}
$\mathrm{BA}_{\kappa}$: up to isomorphism, there is a unique $\kappa$-dense
suborder of $\R$.
\end{center}

We know that $\mathrm{BA}_{\aleph_0}$ holds in ZFC, while the consistency of
$\mathrm{BA}_{\aleph_1}$ with ZFC was proved in \cite{Baumga1973}. Moreover,
$\mathrm{BA}_{\aleph_1}$ follows from PFA. The interest for
$\mathrm{BA}_{\aleph_2}$ was rekindled recently, as witnessed by \cite{MT}.
Itay Neeman recently announced a proof of the consistency of
$\mathrm{BA}_{\aleph_2}$ from large cardinals.

\begin{theorem}\label{underPFA}
Let $\kappa $ be an uncountable cardinal smaller than the continuum. Assume
$\mathrm{BA}_{\kappa }$. Then there exist strongly surjective orders of
cardinality $\kappa $.
\end{theorem}

We prove in fact the following.

\begin{proposition} \label{strsurR}
Let $\kappa$ be an uncountable cardinal smaller than the continuum. Suppose
that, up to isomorphism, $X$ is the unique $\kappa$-dense suborder of $\R$.
Then $L\strong X$ for any $L \subseteq \R$ with $|L|\leq \kappa$.
\end{proposition}
\begin{proof}
Let $M\subseteq L$ be a countable dense set in the order topology of $L$
containing all points having an immediate successor or an immediate
predecessor in $L$ and the endpoints of $L$ if they exist. Let $Y=\sum_{l\in
L}Y_l$, where $Y_l$ is isomorphic to $X$ if $l\in M$, and a singleton
otherwise. It is then enough to show that $Y$ is isomorphic to $X$. It is
easily checked that $Y$ is $\kappa $-dense, so it remains to check that $Y
\inj \R$.

For each $l\in M$ let $Z_l \simeq \R$ be a linear order containing $Y_l$. If
$l\in L\setminus M$, let $Z_l=Y_l$. Finally, let $Z$ be the completion of
$\sum_{l\in L}Z_l$. The proof will be concluded by showing that $Z \simeq
\R$.

By construction, $Z$ does not have minimum nor maximum and it is complete. To
apply the classical characterization of the order type of $\R$ (\cite{Rosens1982}, Theorem
2.30) it remains to show that $Z$ is separable. For each $l\in
M$, pick a countable dense subset $Q_l$ of $Z_l$. Then $\sum_{l\in M}Q_l$ is
dense in $Z$.
\end{proof}

Applying Corollary \ref{multiple}, one obtains that if $X$ is the order
provided by $\mathrm{BA}_{\kappa }$, then each $X^n$ is a strongly surjective
order. These are in fact distinct order types, more precisely the following
holds.

\begin{proposition}
Assume $\mathrm{BA}_{\kappa }$ and let $X$ witness it. Then
\[
X<_iX^2<_iX^3<_i\ldots\mbox{ and }X<_sX^2<_sX^3<_s\ldots
\]
\end{proposition}
\begin{proof}
Since $X^n \strong X^{n+1}$, it is enough to show inductively $X^{n+1}
\not\inj X^n$. First notice that $X^2 \not\inj X$, because
$X^2$ contains an uncountable family of pairwise
disjoint open intervals, while $X$ does not.

Now assume
$X^{n+1} \not\inj X^n$, and suppose towards a contradiction that $f: X^n
\times X^2 \to X^n \times X$ witnesses $X^{n+2} \inj X^{n+1}$. Write
$f(a,x)=(g(a,x), h(a,x))$. If for some $a \in X^n$ the map $x \mapsto h(a,x)$
is injective then it witnesses $X^2 \inj X$, contradicting the base step of
our induction. Hence for all $a \in X^n$ there exist distinct $x,y \in X^2$
such that $h(a,x) = h(a,y)$. Fix $a \in X^n$, pick $x=(x_0,x_1)$ and
$y=(y_0,y_1)$ with this property, and such that $x <_{X^2} y$. For all $z \in
(x,y)_{X^2}$ and $b \in X^n$ we have $(a,x) <_{X^{n+2}} (b,z) <_{X^{n+2}}
(a,y)$ and hence $h(b,z) = h(a,x)$. This implies that $(b,z) \mapsto g(b,z)$
is injective and witnesses $X^n \times (x,y)_{X^2} \inj X^n$. Now notice that
$X \inj (x,y)_{X^2}$: this is clear if $x_1 <_X y_1$, while when $x_1 = y_1$
it follows from the fact that $(x_0,y_0)_X$ is $\kappa$-dense, and hence
isomorphic to $X$. Putting all together, we have shown that $X^{n+1} \inj
X^n$, against the induction hypothesis.
\end{proof}

In Theorem \ref{underPFA} we proved the existence of uncountable strongly
surjective orders under a consequence of PFA. So it is natural to try to
build strongly surjective orders under quite orthogonal principles, like
$\Diamond$ and its variations. To this end it appears that the following
notions, with roots in \cite{baumgartner1981}, are relevant.

\begin{definition}
A partial order $(T,{\leq_T})$ (often denoted only by $T$) is a \emph{tree}
if for all $t \in T$ the initial interval $({\lao}, t[_T$ has order-type an
ordinal $\alpha$ called the \emph{length} of $t$. The set of nodes of length
$\alpha$ is the \emph{$\alpha$th level} of $T$: we denote it by $T^\alpha$.
The \emph{height} of $T$ is the smallest ordinal $\alpha$ such that
$T^\alpha$ is empty. Moreover we say that $s,t \in T$ belong to the same
\emph{brotherhood} of $T$ if $({\lao}, t[_T\, = ({\lao}, s[_T$.

If every brotherhood of $T$ is linearly ordered, then we can order $T$ by the
lexicographic order we denote by $\leq_{lex}$ and, following Baumgartner
\cite{baumgartner1981}, we say that $T$ is \emph{doubly ordered}.

Two doubly ordered trees $T$ and $S$ are isomorphic (and we write $(T,
{\leq_T}, {\leq_{lex}}) \cong (S, {\leq_S}, {\leq_{lex}})$) if there exists a
bijection $f:T \to S$ which preserves both the partial and the linear orders.
\end{definition}

\begin{definition}\label{def:Baumtr}
A doubly ordered tree $(T,{\leq_T}, {\leq_{lex}})$ is a \emph{Baumgartner
tree} if the following conditions hold:
\begin{itemize}
\item $(T,{\leq_T})$ is a \emph{Suslin} tree (that is, every chain and
    every antichain in $T$ is countable, the height of $T$ is $\omega_1$,
    and for all $t \in T^\alpha$ and all $\beta<\omega_1$ with $\alpha \leq
    \beta$ there exists $s \in T^\beta$ such that $t \leq_T s$).
\item $T$ has \emph{rational brotherhoods}, that is the ordering of each
    brotherhood of $T$ is isomorphic to $\Q$.
\item for every $X \in [T]^{\aleph_1}$, if
\begin{itemize}
 \item $X$ is cofinal with respect to $\leq_T$ in $\{t \in T \mid \exists
     s \in X (s \leq_T t)\}$,
 \item $X$ has a rational basis, that is  the ordering of its minimal
     elements is isomorphic to $\Q$,
\end{itemize}
then we have $(X, {\leq_T}, {\leq_{lex}}) \cong (T, {\leq_T},
{\leq_{lex}})$.
\end{itemize}
\end{definition}

Assuming $\Diamond^+$, in Theorem 4.15 of \cite{baumgartner1981}, Baumgartner
claimed to build a minimal Specker type which was in fact the linear part of
what we call a Baumgartner tree. As pointed out by Hossein Lamei Ramandi,
Baumgartner's proof has however a gap: there are indeed many counterexamples
to the crucial Lemma 4.14, stated without proof.

Recently, D\'{a}niel Soukup (\cite{Soukup}, see Section 4) modified the proof
of Theorem 2.3 in \cite{HNS} to construct a Baumgartner tree under
$\Diamond^+$.

\begin{theorem}\label{prop:underdiamond}
The linear order of a Baumgartner tree is strongly surjective.
\end{theorem}
\begin{proof}
Fix a Baumgartner tree $(T, {\leq_T}, {\leq_{lex}})$ and let $L = (T,
\leq_{lex})$.

Observe first that, as $T$ has rational brotherhoods, $(T^0, \leq_{lex})
\simeq \Q$ holds. Hence $L$ can be written as a $\Q$-sum and $\Q \strong L$.
For any countable linear order $K$ we have $K \strong \Q$ and hence $K
\strong L$.

It remains to deal with the uncountable suborders of $L$. Given $X \sub T$
uncountable let $K= (X,\leq_{lex})$. Consider
\[
X^\circ=\{s \in X\mid X \text{ is cofinal in } \{t \in T \mid s \leq_T t\}\},
\]
and call $A$ the set of its $\leq_T$-minimal elements. As $A$ is an antichain
and $T$ is Suslin, $A$ is countable. Let $B = X\setminus X^\circ$: $B$ is not
cofinal in any set of the form $\{t \in T \mid s \leq_T t\}$ and hence, by a
well-known property of Suslin trees, it is countable. For $s \in B$ let $L_s$
be the singleton order, and for $s \in A$ let $L_s = 1+L$. Notice that for $s
\in A$ we have $L_s \simeq (\{t \in X \mid s \leq_T t\}, {\leq_{lex}})$ by
the last clause in the definition of Baumgartner tree. The set $\{t \in X
\mid s <_T t\}$ has indeed a rational basis $C$: take $t,t'\in C$ with
$t<_{lex}t'$ and call $s'=t\cap t'\geq_T s$. The second clause implies that
there are immediate successors $s_0<_{lex}s_1<_{lex}s_2<_{lex}s_3<_{lex}s_4$
of $s'$ such that $s_1\leq_T t$ and $s_3\leq_T t'$. Cofinality gives us then
$t_i\geq_T s_i$ in $C$ for $i=0,2,4$, so $C$ is densely ordered with no
extremes. As an antichain in $T$, $C$ is countable.

Thus we have $K \simeq \sum_{s \in A \cup B} L_s$, where the sum is taken
according to lexicographic order.
Since the $0$th level of $T$ is ordered as $\Q$ we have $L\simeq L\Q$, so for any countable linear order $M$ we have, as linear orders: $LM\simeq (L\Q)M\simeq L(\Q M)\simeq L\Q\simeq L$.
In particular $L_s\simeq1+L\strong L2 \simeq L$ for every $s \in A$, so we have $\sum_{s \in A \cup B} L_s
\strong LM$, where $M= (A \cup B, \leq_{lex})$. Since $M$ is countable, $LM \simeq L$.
Altogether we showed $K \strong L$, as needed.
\end{proof}

\subsection{Some directions for further research}
In this paper, we have given a fairly complete treatment of countable strongly
surjective linear orders. On the other hand, various problems regarding
uncountable strongly surjective linear orders have been left open. We discuss
briefly here some lines for further research in this direction.

\subsubsection{The existence of uncountable strongly surjective orders}
In the first draft of this paper, we left open three problems concerning strongly surjective orders. The main one
was: Does there exist an uncountable strongly surjective order in ZFC?
With two related questions: Does there exist an uncountable strongly surjective order under CH? Or
under $\Diamond$?

Upon learning about these problems, D\'{a}niel Soukup answered negatively the
second question, so a fortiori the first one: see Section 5 of
\cite{Soukup}.

The existence of a strongly surjective order under $\Diamond$ remains open. D\'{a}niel Soukup included a healthy list of open problems about uncountable strongly surjective orders in Section 6 of \cite{Soukup}.

\subsubsection{Definably strongly surjective orders}
In Theorem \ref{nocantor} we proved that no uncountable Borel suborder of
$2^{ \N }$ can be strongly surjective by using definability reasons: such an
order cannot surject onto a non-analytic suborder, since epimorphisms are
Borel. This suggests that there may be some Borel subsets of $2^\N$ for which
this is the only obstruction to strong surjectivity, as they admit
epimorphisms onto all their analytic suborders. Call such orders
\emph{definably strongly surjective}.

Corollaries \ref{cor:gaps} and \ref{cor:lexprod} (but not Corollary
\ref{noznqn}) show that $2^\N$ and $\Z^\N$ ordered lexicographically are not
definably strongly surjective.

\begin{question}
Do there exist definably strongly surjective orders that are not strongly
surjective? In particular, is $\Q^\N$ definably strongly surjective? Can the
concept of a definably strongly surjective order be extended beyond the Borel
suborders of $2^{ \N }$?
\end{question}

\section*{Acknowledgments.}
An early version of this work, and notably a boldface form of Theorem
\ref{theorem:boreluniformization}, was presented to the \lq\lq Groupe de
travail en th\'eorie descriptive des ensembles\rq\rq\ at University Paris 6.
We would like to thank the whole group for their patience, and Alain Louveau
for suggesting the lightface version of Theorem
\ref{theorem:boreluniformization}, and pointing out that it would yield
Corollary \ref{cor:scatstrongsurj}.

The research presented in this paper has been done while the first author was
visiting the Department of information systems of the University of Lausanne.
He wishes to thank the \'Equipe de logique, and in particular its director
prof. Jacques Duparc, for providing such a friendly environment.

The research of the second author was funded by a fellowship from the
Istituto Nazionale d'Alta Matematica (INdAM) and in part by FWF Grant P28153.

The research of the third author was supported by PRIN 2012 Grant \lq\lq
Logica, Modelli e Insiemi\rq\rq.

The paper was completed while all three authors were attending the workshop
\lq\lq Current trends in Descriptive Set Theory\rq\rq\ at the Erwin
Schr\"{o}dinger International Institute for Mathematics and Physics in
Vienna. The authors wish to thank the ESI for its support and hospitality.

Finally, we thank the anonymous referee for numerous remarks and suggestions.

\end{document}